\documentclass[11pt]{article}
\usepackage[a4paper,margin=27mm]{geometry}
\usepackage{amsmath,amssymb,amsthm,mathtools}
\usepackage{ytableau}
\usepackage[table]{xcolor}
\usepackage{microtype}
\usepackage{needspace}
\usepackage{booktabs,array,tabularx}
\usepackage{tikz}
\usetikzlibrary{arrows.meta,calc,positioning}
\usepackage{float}
\usepackage{placeins}
\usepackage[hidelinks]{hyperref}

\hypersetup{
  pdftitle={Crystal operators and a sign-reversing involution on semistandard set-valued tableaux of skew shape},
  pdfauthor={Tatsushi Shimazaki}
}

\numberwithin{equation}{section}
\newtheorem{theorem}{Theorem}[section]
\newtheorem{proposition}[theorem]{Proposition}
\newtheorem{lemma}[theorem]{Lemma}
\newtheorem{corollary}[theorem]{Corollary}
\theoremstyle{definition}
\newtheorem{definition}[theorem]{Definition}
\newtheorem{example}[theorem]{Example}
\theoremstyle{plain}

\newcommand{\SST}{\operatorname{SST}}
\newcommand{\SVT}{\operatorname{SVT}}
\newcommand{\DFV}{\operatorname{DFV}}
\newcommand{\wt}{\operatorname{wt}}
\newcommand{\ex}{\operatorname{ex}}
\newcolumntype{Y}{>{\centering\arraybackslash}X}
\definecolor{mutedblue}{RGB}{92,132,176}
\definecolor{mutedred}{RGB}{186,112,122}
\definecolor{mutedgreen}{RGB}{108,150,104}
\definecolor{mutedpurple}{RGB}{126,95,188}
\definecolor{pathgray}{RGB}{85,85,85}
\definecolor{faintgray}{RGB}{220,220,220}
\definecolor{softgray}{RGB}{205,205,205}
\newcommand{\localvertex}[2]{%
\begin{tikzpicture}[baseline=-0.6ex,x=0.40cm,y=0.40cm,line width=0.7pt,>=Stealth]
  \draw[softgray] (-1.30,0)--(1.30,0);
  \draw[softgray] (0,-1.30)--(0,1.30);
  #1
  #2
\end{tikzpicture}}

\ytableausetup{mathmode,boxsize=1.55em,centertableaux}

\begin{document}

\begin{center}
{\Large\bfseries Crystal operators and a sign-reversing involution\\
on semistandard set-valued tableaux of skew shape}\par
\vspace{1mm}
{\large Tatsushi Shimazaki\footnotemark}\par
\end{center}
\footnotetext{Liberal Arts, National Institute of Technology, Akashi College, Akashi, Hyogo 674-8501, Japan.\\
E-mail: \href{mailto:t.shimazaki@akashi.ac.jp}{t.shimazaki@akashi.ac.jp}.
}
\vspace{3mm}

\begin{abstract}
We prove commutation relations between the ordinary type~A crystal operators and a sign-reversing involution on semistandard set-valued tableaux of skew shape. We describe the involution by Boolean fibers with fixed maximum entries. It commutes with every raising operator except the one indexed by the entry changed by the involution. We characterize commutation at this exceptional raising operator and for lowering operators away from the two colors involving that entry using the box where the involution acts and reduced signatures. We construct a bijection from these tableaux to decorated five-vertex states that preserves weight and excess, subject to a restriction on nontrivial bumps determined by the inner partition. Rotation with alphabet reversal yields the dual statements.
\end{abstract}

\section{Introduction}
\label{sec:introduction}

Semistandard set-valued tableaux were introduced by Buch and yield tableau formulas for stable Grothendieck polynomials \cite{Buch2002}. Combinatorial relations between skew Schur and skew stable Grothendieck polynomials are studied in \cite{ChanPflueger2021}. Related tableau models for Grothendieck polynomials are compared in \cite{Hawkes2024}. Grothendieck polynomials were introduced by Lascoux and Sch\"utzenberger in connection with the $K$-theory of flag varieties \cite{LS1982}. Their Grassmannian specializations were studied combinatorially by Lenart \cite{Lenart2000}. Factorial Grothendieck polynomials are treated in \cite{McNamara2006}.

A sign-reversing involution on semistandard set-valued tableaux of skew shape was defined in \cite{FNSOddness}. Crystal structures for canonical Grothendieck functions are studied in \cite{HawkesScrimshaw2020}. Bender--Knuth-type involutions in $K$-theoretic tableau combinatorics are used in \cite{IkedaShimazaki2014} to prove $K$-theoretic Littlewood--Richardson rules. Crystal bases were introduced by Kashiwara \cite{Kashiwara1990,Kashiwara1991}, and tableau crystal models for classical types were developed by Kashiwara and Nakashima \cite{KashiwaraNakashima1994}. Flagged refined skew stable Grothendieck polynomials are studied through Demazure crystals in \cite{Kundu2026}. For semistandard set-valued tableaux, the ordinary type~A crystal operators were defined for straight shapes in \cite{MPS2021} and extended to skew shapes in \cite{MPPS2020}. $K$-theoretic crystals on set-valued tableaux of rectangular shape are constructed in \cite{PechenikScrimshaw2022}. Crystal operators preserve tableau excess. Every nontrivial application of the sign-reversing involution changes tableau excess by one. This paper considers the interaction between these crystal operators and the involution.

The involution and the decomposition by prescribed maximum entries are taken from \cite{FNSOddness}. The ordinary type~A crystal operators are those of \cite{MPS2021,MPPS2020}. The marked Gelfand--Tsetlin and decorated five-vertex constructions for straight shapes are based on \cite{JNS2026,MPS2021}. The commutation criteria appear in Section~\ref{sec:iota-crystal}, their dual forms obtained by rotation in Section~\ref{sec:duality}, and the realizations for skew shapes by marked Gelfand--Tsetlin patterns and decorated five-vertex states in Section~\ref{sec:five-vertex}. The involution organizes cancellation within fibers determined by maximum entries, and the crystal operators organize the same tableaux into type~A crystal graphs. The commutation criteria determine where these two structures are compatible.

For a fixed maximum tableau, the independent choices form a graded Boolean lattice whose rank is tableau excess (Proposition~\ref{prop:max-fiber-boolean}). Its weighted generating function factors over the Boolean coordinates. On every nontrivial orbit, the involution changes one of these coordinates and traverses one edge of the Boolean lattice (Proposition~\ref{prop:iota-coordinate}).

For $T\ne T_0$, the involution $\iota$ of \cite{FNSOddness} acts at the first discrepancy of $T$, the first box in column order whose content differs from the fixed singleton tableau $T_0$. It inserts or removes the entry $k_0$ of $T_0$ in that box. The first discrepancy controls the interaction with the crystal operators. A raising operator cannot select a box preceding the first discrepancy. For every $r\ne k_0$, one has $e_r\iota=\iota e_r$ (Theorem~\ref{thm:raising-away}). For $r\notin\{k_0-1,k_0\}$, the equality $f_r\iota(T)=\iota f_r(T)$ holds if and only if $f_r(T)=0$ or the box selected by $f_r$ does not precede the first discrepancy (Proposition~\ref{prop:lowering-generic}). At $r=k_0$, the involution changes the reduced $k_0$-signature. The remaining commutation problem is characterized by the first discrepancy after applying $e_{k_0}$ and the occurrence selected in the reduced signature (Theorem~\ref{thm:exceptional}). The lowering colors $k_0-1$ and $k_0$ are not characterized in this paper.

Rotation of the skew diagram together with reversal of the alphabet exchanges maximum and minimum entries (Proposition~\ref{prop:rotation}). Under this transformation, a fiber determined by maximum entries becomes a fiber determined by minimum entries with the same Boolean structure. Conjugation by this transformation defines the dual involution $\iota^*$, which preserves minimum entries (Proposition~\ref{prop:iota-star}). Rotation with alphabet reversal exchanges raising and lowering crystal operators, and the commutation results for the dual involution follow from those for the original involution. The dual raising colors are likewise not characterized.

The crystal operators and both involutions also act on decorated five-vertex states (see Section~\ref{sec:five-vertex}). Colored lattice models for double Grothendieck polynomials occur in \cite{BuciumasScrimshaw2022}, and colored five-vertex models for Lascoux polynomials in \cite{BSW2020}. Decorated five-vertex states and crystal transformations for straight shapes are developed in \cite{JNS2026}. Marked Gelfand--Tsetlin patterns for semistandard set-valued tableaux of straight shape occur in \cite{MPS2021}. Related five-vertex models in the Grothendieck setting are studied in \cite{Motegi2021,MotegiSakai2013,MotegiSakai2014}. For a skew shape, the inner partition determines the lower boundary and imposes an additional condition on the marks and nontrivial bumps. Semistandard set-valued tableaux, marked Gelfand--Tsetlin patterns, and decorated five-vertex states are in bijection with preservation of weight and excess. The bijection intertwines the crystal operators on decorated states with the ordinary type~A crystal operators (Theorem~\ref{thm:crystal-intertwining}). On decorated states, $\iota$ acts by a finite displacement of a path segment (Proposition~\ref{prop:iota-mgt}). The involution $\iota^*$ preserves the interlacing chain and changes one bump decoration (Proposition~\ref{prop:iota-star-mgt}).

Section~\ref{sec:tableaux} contains the preliminaries. Sections~\ref{sec:max-fibers}--\ref{sec:duality} develop the theory of the involutions and crystal operators on tableaux. Section~\ref{sec:five-vertex} presents the marked Gelfand--Tsetlin and decorated five-vertex realizations. Section~\ref{sec:conclusion} concludes the paper.

\section{Preliminaries}
\label{sec:tableaux}

\subsection{Set-valued tableaux}
\label{subsec:svt-definitions}

A partition
\[
\lambda=(\lambda_1,\lambda_2,\ldots)
\]
is a weakly decreasing sequence of nonnegative integers with
$\lambda_i=0$ for all sufficiently large $i$. The partition $\lambda$ is identified
with its Young diagram in English notation,
\[
\lambda=\{(i,j)\in\mathbb Z_{>0}^2:1\le j\le\lambda_i\}.
\]
Its size, length, and conjugate are
\[
|\lambda|=\sum_{i\ge1}\lambda_i,
\qquad
\ell(\lambda)=|\{i:\lambda_i>0\}|,
\qquad
\lambda'_j=|\{i:\lambda_i\ge j\}|.
\]

For partitions $\mu\subseteq\lambda$, the skew diagram
$\lambda/\mu$ is the set difference of their Young diagrams.
A box is denoted by $u=(i,j)$. Two boxes are horizontally adjacent
if they share a vertical edge, and vertically adjacent if they share
a horizontal edge.

Fix partitions $\mu\subseteq\lambda$ and an integer
$n\in\mathbb Z_{>0}$. All tableau entries lie in
$\{1,\ldots,n\}$.

A semistandard tableau of shape $\lambda/\mu$ is a filling by entries
in $\{1,\ldots,n\}$ that is weakly increasing along rows and strictly
increasing down columns. The set of these tableaux is denoted by
$\SST_n(\lambda/\mu)$.

For a filling $T$ by nonempty subsets of $\{1,\ldots,n\}$, the
subset in a box $u=(i,j)$ is denoted by $T_u$ or $T_{i,j}$. Set
\[
T_u^{\min}=\min T_u,
\qquad
T_u^{\max}=\max T_u.
\]
A singleton is identified with its unique entry.

\begin{definition}[\cite{Buch2002}]
\label{def:svt}
A \emph{semistandard set-valued tableau} of shape $\lambda/\mu$
is a filling $T$ by nonempty subsets of $\{1,\ldots,n\}$ satisfying
\[
T_{i,j}^{\max}\le T_{i,j+1}^{\min}
\]
for every horizontally adjacent pair $(i,j),(i,j+1)$, and
\[
T_{i,j}^{\max}<T_{i+1,j}^{\min}
\]
for every vertically adjacent pair $(i,j),(i+1,j)$.
\end{definition}

The set of semistandard set-valued tableaux of shape $\lambda/\mu$
is denoted by $\SVT_n(\lambda/\mu)$. The set
$\SST_n(\lambda/\mu)$ is identified with the subset of
$\SVT_n(\lambda/\mu)$ whose box contents are singletons. Throughout,
$\SVT_n(\lambda/\mu)$ is assumed to be nonempty.

For $T\in\SVT_n(\lambda/\mu)$, an entry of $T_u$ other than
$T_u^{\min}$ is an \emph{extra entry}. A box containing more than one
entry is a \emph{multicell}. The \emph{excess} of $T$ is
\[
\ex(T)=\sum_{u\in\lambda/\mu}(|T_u|-1).
\]
For $1\le k\le n$, let
\[
m_k(T)=|\{u\in\lambda/\mu:k\in T_u\}|.
\]
The weight of $T$ and its monomial are
\[
\wt(T)=(m_1(T),\ldots,m_n(T)),
\qquad
x^{\wt(T)}=\prod_{k=1}^n x_k^{m_k(T)}.
\]

The skew stable Grothendieck polynomial is
\begin{equation}
\label{eq:grothendieck}
G_{\lambda/\mu}(x_1,\ldots,x_n;\beta)
=
\sum_{T\in\SVT_n(\lambda/\mu)}
\beta^{\ex(T)}x^{\wt(T)}.
\end{equation}
Set-valued tableau formulas originate in \cite{Buch2002}. The parameter $\beta$ follows the convention of \cite{FK1994}. At $\beta=0$,
equation~\eqref{eq:grothendieck} reduces to
\[
G_{\lambda/\mu}(x_1,\ldots,x_n;0)
=
s_{\lambda/\mu}(x_1,\ldots,x_n).
\]
At $\beta=-1$, the contribution of $T$ has sign
$(-1)^{\ex(T)}$.

\begin{definition}
\label{def:minmax-tableaux}
For $T\in\SVT_n(\lambda/\mu)$, its \emph{minimum tableau} $P$ and
\emph{maximum tableau} $Q$ are
\[
P=(T_u^{\min})_{u\in\lambda/\mu},
\qquad
Q=(T_u^{\max})_{u\in\lambda/\mu}.
\]
\end{definition}

For horizontally adjacent boxes,
\[
T_{i,j}^{\min}
\le T_{i,j}^{\max}
\le T_{i,j+1}^{\min}
\le T_{i,j+1}^{\max},
\]
and for vertically adjacent boxes,
\[
T_{i,j}^{\min}
\le T_{i,j}^{\max}
< T_{i+1,j}^{\min}
\le T_{i+1,j}^{\max}.
\]
Hence
\[
P,Q\in\SST_n(\lambda/\mu).
\]

\begin{example}
\label{ex:running}
For $\lambda=(5,4,3)$, $\mu=(2,1)$, and $n=4$, consider
\[
T=
\begin{ytableau}
\none&\none&1&12&2\\
\none&1&23&3\\
1&23&4
\end{ytableau}.
\]
This tableau belongs to $\SVT_4(\lambda/\mu)$. Its multicells are
$(1,4)$, $(2,3)$, and $(3,2)$. Its excess, weight, and monomial are
\[
\ex(T)=3,
\qquad
\wt(T)=(4,4,3,1),
\qquad
x^{\wt(T)}=x_1^4x_2^4x_3^3x_4.
\]
Its minimum and maximum tableaux are
\[
P=
\begin{ytableau}
\none&\none&1&1&2\\
\none&1&2&3\\
1&2&4
\end{ytableau}
\qquad
Q=
\begin{ytableau}
\none&\none&1&2&2\\
\none&1&3&3\\
1&3&4
\end{ytableau}.
\]
\end{example}

\subsection{Ordinary type~A crystal operators}
\label{subsec:tableau-crystal}

The ordinary type~A crystal operators on semistandard set-valued tableaux of straight shape were introduced in \cite{MPS2021}. The sign convention of \cite{MPS2021} is used. For tableaux of skew shape, the extension in \cite[Definition~10 and Theorem~11]{MPPS2020} applies. In \cite{MPPS2020}, the symbols $+$ and $-$ are interchanged. The pairing and operators are unchanged.

Fix $1\le r<n$ and let $T\in\SVT_n(\lambda/\mu)$. Read the columns
of $T$ from left to right. A column containing $r$ but not $r+1$
contributes $+$. A column containing $r+1$ but not $r$
contributes $-$. A column containing both values or neither value
contributes no sign. Starting with the rightmost $-$ and proceeding
from right to left, pair each $-$ with the nearest unpaired $+$ to
its right. After all possible pairings have been made, the unpaired
signs have the form
\[
+\cdots+\,-\cdots-.
\]
They form the \emph{reduced $r$-signature} of $T$. Since a fixed value
occurs at most once in a tableau column, each sign determines a unique
occurrence of $r$ or $r+1$.

\begin{definition}[\cite{MPPS2020}]
\label{def:svt-crystal}
If no unpaired $-$ remains, set $e_r(T)=0$. Otherwise, $e_r$ selects
the occurrence of $r+1$ in the leftmost unpaired $-$ column. If the
box immediately to the left of the selected box belongs to
$\lambda/\mu$ and contains both $r$ and $r+1$, remove $r+1$ from that
left box and insert $r$ into the selected box. Otherwise, replace the
selected entry $r+1$ by $r$.

If no unpaired $+$ remains, set $f_r(T)=0$. Otherwise, $f_r$ selects
the occurrence of $r$ in the rightmost unpaired $+$ column. If the
box immediately to the right of the selected box belongs to
$\lambda/\mu$ and contains both $r$ and $r+1$, remove $r$ from that
right box and insert $r+1$ into the selected box. Otherwise, replace
the selected entry $r$ by $r+1$.
\end{definition}

We call the first case in each nonzero operator, involving the adjacent box, a \emph{transfer}, and the second a \emph{replacement}. In a transfer, the adjacent box from which an entry is removed is the \emph{transfer source}, and the condition that triggers this case is the \emph{transfer condition}.

For the null element, set $e_r(0)=f_r(0)=0$.

\begin{example}
\label{ex:crystal-two-boxes}
For $r=1$, let
\[
T=
\begin{ytableau}
1&2
\end{ytableau}.
\]
The first column contributes $+$ and the second contributes $-$.
Neither sign is paired. The operator $f_1$ selects the entry $1$ in the first
box, and $e_1$ selects the entry $2$ in the second box. \[
f_1(T)=
\begin{ytableau}
2&2
\end{ytableau},
\qquad
e_1(T)=
\begin{ytableau}
1&1
\end{ytableau}.
\]
\end{example}

\subsection{The sign-reversing involution}
\label{subsec:iota}

For $(i,j)\in\lambda/\mu$, the integer $i-\mu'_j$ is the position of
the box in its column of the skew diagram, counted from the top. Define
the tableau $T_0$ by
\[
(T_0)_{i,j}=\{i-\mu'_j\}.
\]
Since $\SVT_n(\lambda/\mu)$ is nonempty, every column of
$\lambda/\mu$ has at most $n$ boxes. The inequality
$\mu'_j\ge\mu'_{j+1}$ implies that the entries of $T_0$ are weakly
increasing along rows. They increase by one down each column and belong to
$\{1,\ldots,n\}$. In particular,
\[
T_0\in\SST_n(\lambda/\mu).
\]

The \emph{column order} on the boxes of $\lambda/\mu$ is the order by
increasing column index and, within each column, by increasing row index.
For $T\ne T_0$, let
$u_\iota(T)$ be the first box in this order whose content differs
from that of $T_0$. The box $u_\iota(T)$ is called the \emph{first discrepancy}
of $T$. If
\[
u_\iota(T)=(i,j),
\]
the entry of $T_0$ in this box is
\[
k_0=i-\mu'_j.
\]

\begin{definition}[\cite{FNSOddness}]
\label{def:iota}
For $T\ne T_0$ with first discrepancy $u=u_\iota(T)$, define
\[
(\iota(T))_v=
\begin{cases}
T_v,&v\ne u,\\[1mm]
T_u\cup\{k_0\},&v=u,\ k_0\notin T_u,\\[1mm]
T_u\setminus\{k_0\},&v=u,\ k_0\in T_u.
\end{cases}
\]
\end{definition}

Extend this map to all of
$\SVT_n(\lambda/\mu)$ by setting
\[
\iota(T_0)=T_0.
\]
For compositions with the crystal operators, set
\[
\iota(0)=0.
\]

For $T\ne T_0$, \cite[Appendix~B]{FNSOddness} proves that
$\iota(T)\in\SVT_n(\lambda/\mu)$, that $\iota$ is a
fixed-point-free involution on
$\SVT_n(\lambda/\mu)\setminus\{T_0\}$, and that the first discrepancy
of $\iota(T)$ is again $u_\iota(T)$. These properties are also recovered
from the description in terms of Boolean fibers in Proposition~\ref{prop:iota-coordinate}.
As a map on $\SVT_n(\lambda/\mu)$, the extended involution has the unique
fixed point $T_0$. For $T\ne T_0$, one application inserts or removes
one entry, and
\[
\ex(\iota(T))=\ex(T)\pm1.
\]
\[
(-1)^{\ex(\iota(T))}=-(-1)^{\ex(T)}.
\]

\begin{example}
\label{ex:iota-on-skew-svt}
For the tableau in Example~\ref{ex:running},
\[
T_0=
\begin{ytableau}
\none&\none&1&1&1\\
\none&1&2&2\\
1&2&3
\end{ytableau}.
\]
Its first discrepancy is $u_\iota(T)=(3,2)$, with $k_0=2$, and
\[
\iota(T)=
\begin{ytableau}
\none&\none&1&12&2\\
\none&1&23&3\\
1&3&4
\end{ytableau}.
\]
The first discrepancy of $\iota(T)$ is again $(3,2)$.
\end{example}

\section{Maximum fibers and Boolean coordinates}
\label{sec:max-fibers}

Appendix~C of \cite{FNSOddness} decomposes
$\SVT_n(\lambda/\mu)$ according to the maximum tableau and describes
the entries smaller than each prescribed maximum by independent choices.
These choices form Boolean coordinates. Tableau excess counts the
selected coordinates, and the involution changes one coordinate on
every nontrivial orbit.

\subsection{Fibers determined by the maximum tableau}
\label{subsec:max-fibers}

For straight shapes, parallel decompositions according to prescribed
maximum and minimum entries appear in
\cite[Proposition~8.11]{MoralesZhu2022}.

For $Q\in\SST_n(\lambda/\mu)$, define
\[
\mathcal F_Q^{\max}
=
\{T\in\SVT_n(\lambda/\mu):
  T_u^{\max}=Q_u
  \text{ for every }u\in\lambda/\mu\}.
\]
The singleton tableau $Q$ belongs to $\mathcal F_Q^{\max}$.

Let $u=(i,j)\in\lambda/\mu$. For a fixed maximum tableau $Q$, every
entry $k<Q_{i,j}$ occurring in the box $(i,j)$ satisfies the row and
column restrictions imposed by the prescribed maxima immediately to
the left and above. Define
\begin{equation}
\label{eq:omega-Q}
\begin{split}
\Omega(Q)=\{((i,j),k):{}&
(i,j)\in\lambda/\mu,\quad 1\le k<Q_{i,j},\\
&Q_{i,j-1}\le k
\quad\text{if }(i,j-1)\in\lambda/\mu,\\
&Q_{i-1,j}<k
\quad\text{if }(i-1,j)\in\lambda/\mu
\}.
\end{split}
\end{equation}
An element $((i,j),k)\in\Omega(Q)$ records an optional entry
$k<Q_{i,j}$ in the box $(i,j)$.

For the tableau $T_0$ from Subsection~\ref{subsec:iota},
\[
\Omega(T_0)=\varnothing.
\]
If the entry of $T_0$ in a box is $k>1$, the box immediately above
it belongs to the skew diagram and has entry $k-1$. The last
condition in \eqref{eq:omega-Q} excludes every integer smaller than
$k$. At the top of a column, the entry of $T_0$ is $1$.

\begin{example}
\label{ex:max-fiber-coordinates}
For the maximum tableau in Example~\ref{ex:running},
\[
Q=
\begin{ytableau}
\none&\none&1&2&2\\
\none&1&3&3\\
1&3&4
\end{ytableau},
\]
one has
\[
\Omega(Q)
=
\{((1,4),1),\ ((2,3),2),\ ((3,2),2)\}.
\]
The three optional entries occur in the boxes $(1,4)$, $(2,3)$,
and $(3,2)$.
\end{example}

\subsection{Boolean order and weighted fiber sums}
\label{subsec:max-fiber-boolean}

The independent coordinates in \cite[Appendix~C]{FNSOddness} carry
the order induced by inclusion in each box. For $T,U\in\mathcal F_Q^{\max}$,
write $T\le U$ if
\[
T_u\subseteq U_u
\qquad
\text{for every }u\in\lambda/\mu.
\]
The Boolean lattice $2^{\Omega(Q)}$ is ordered by inclusion.

\begin{proposition}
\label{prop:max-fiber-boolean}
For $Q\in\SST_n(\lambda/\mu)$, the map
\[
\Phi_Q:\mathcal F_Q^{\max}\longrightarrow 2^{\Omega(Q)},
\qquad
\Phi_Q(T)
=
\{(u,k)\in\Omega(Q):k\in T_u\},
\]
is an isomorphism of graded posets. Under this isomorphism,
\[
\operatorname{rank}(T)=\ex(T).
\]
In particular,
\[
|\mathcal F_Q^{\max}|=2^{|\Omega(Q)|}.
\]
\end{proposition}

\begin{proof}
Let $u=(i,j)$. If $k<Q_{i,j}$ occurs in $T_u$ for
$T\in\mathcal F_Q^{\max}$, semistandardness implies
\[
Q_{i,j-1}\le k
\]
if $(i,j-1)\in\lambda/\mu$, and
\[
Q_{i-1,j}<k
\]
if $(i-1,j)\in\lambda/\mu$. These inequalities place
$(u,k)$ in $\Omega(Q)$.

Conversely, for $B\subseteq\Omega(Q)$, set
\[
T_u=\{Q_u\}\cup\{k:(u,k)\in B\}.
\]
For horizontally adjacent boxes $(i,j)$ and $(i,j+1)$, every optional
entry in the right box is at least $Q_{i,j}$ by
\eqref{eq:omega-Q}, and
\[
T_{i,j}^{\max}=Q_{i,j}\le T_{i,j+1}^{\min}.
\]
For vertically adjacent boxes $(i,j)$ and $(i+1,j)$, every optional
entry in the lower box is strictly larger than $Q_{i,j}$, and
\[
T_{i,j}^{\max}=Q_{i,j}<T_{i+1,j}^{\min}.
\]
This filling belongs to $\mathcal F_Q^{\max}$.

The two constructions are inverse. Under $\Phi_Q$, inclusion in every box is subset inclusion in $\Omega(Q)$. Each selected coordinate
contributes one entry in addition to the prescribed maximum:
\[
\ex(T)=|\Phi_Q(T)|.
\]
\end{proof}

\begin{corollary}
\label{cor:max-fiber-factorization}
For every $Q\in\SST_n(\lambda/\mu)$,
\begin{equation}
\label{eq:max-fiber-factorization}
\sum_{T\in\mathcal F_Q^{\max}}
\beta^{\ex(T)}x^{\wt(T)}
=
x^{\wt(Q)}
\prod_{(u,k)\in\Omega(Q)}(1+\beta x_k).
\end{equation}
\begin{equation}
\label{eq:grothendieck-max-fibers}
G_{\lambda/\mu}(x_1,\ldots,x_n;\beta)
=
\sum_{Q\in\SST_n(\lambda/\mu)}
x^{\wt(Q)}
\prod_{(u,k)\in\Omega(Q)}(1+\beta x_k).
\end{equation}
\end{corollary}

\begin{proof}
The singleton tableau $Q$ contributes $x^{\wt(Q)}$. Selecting a
coordinate $(u,k)$ increases the excess by one and contributes one
additional occurrence of $k$. Its contribution is $\beta x_k$.
The independent contributions over $\Omega(Q)$ form the product in
\eqref{eq:max-fiber-factorization}.

Every tableau in $\SVT_n(\lambda/\mu)$ has a unique maximum tableau.
The fibers $\mathcal F_Q^{\max}$ partition $\SVT_n(\lambda/\mu)$.
Equation~\eqref{eq:grothendieck-max-fibers} is the sum of
\eqref{eq:max-fiber-factorization} over all $Q\in\SST_n(\lambda/\mu)$.
\end{proof}

\begin{example}
\label{ex:max-fiber}
For the tableau $Q$ in Example~\ref{ex:max-fiber-coordinates}, the
eight tableaux in $\mathcal F_Q^{\max}$ are arranged by excess.

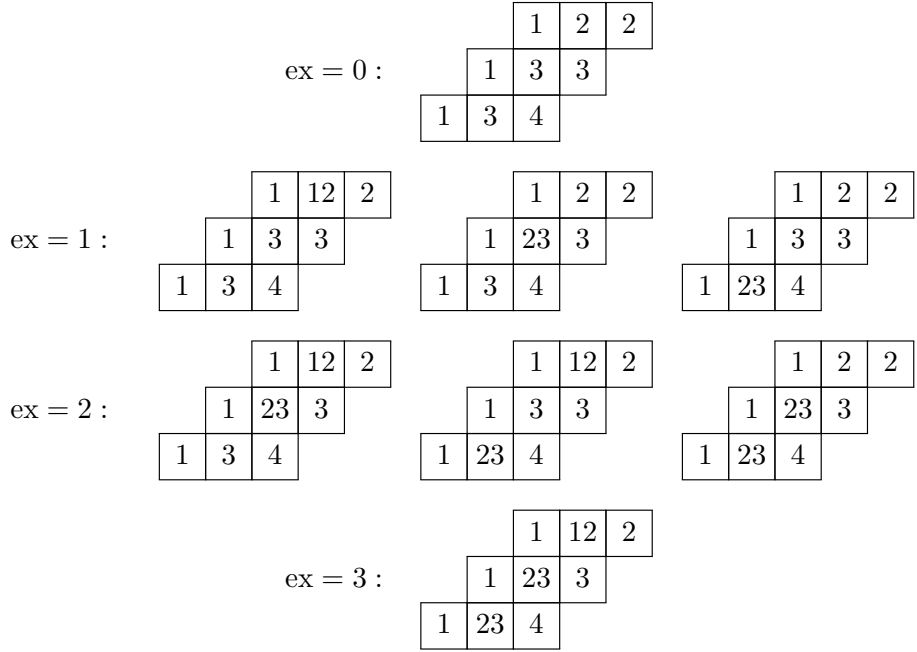
\begingroup
\setlength{\abovedisplayskip}{7pt}
\setlength{\belowdisplayskip}{7pt}
\setlength{\abovedisplayshortskip}{5pt}
\setlength{\belowdisplayshortskip}{5pt}
\ytableausetup{mathmode,boxsize=1.55em,centertableaux}
\[
\ex=0:\quad
\begin{ytableau}
\none&\none&1&2&2\\
\none&1&3&3\\
1&3&4
\end{ytableau}
\]
\[
\ex=1:\quad
\begin{array}{c@{\quad}c@{\quad}c}
\begin{ytableau}
\none&\none&1&12&2\\
\none&1&3&3\\
1&3&4
\end{ytableau}
&
\begin{ytableau}
\none&\none&1&2&2\\
\none&1&23&3\\
1&3&4
\end{ytableau}
&
\begin{ytableau}
\none&\none&1&2&2\\
\none&1&3&3\\
1&23&4
\end{ytableau}
\end{array}
\]
\[
\ex=2:\quad
\begin{array}{c@{\quad}c@{\quad}c}
\begin{ytableau}
\none&\none&1&12&2\\
\none&1&23&3\\
1&3&4
\end{ytableau}
&
\begin{ytableau}
\none&\none&1&12&2\\
\none&1&3&3\\
1&23&4
\end{ytableau}
&
\begin{ytableau}
\none&\none&1&2&2\\
\none&1&23&3\\
1&23&4
\end{ytableau}
\end{array}
\]
\[
\ex=3:\quad
\begin{ytableau}
\none&\none&1&12&2\\
\none&1&23&3\\
1&23&4
\end{ytableau}
\]
$\mathcal F_Q^{\max}$ is a Boolean lattice of rank $3$, with rank sizes
$1,3,3,1$. Formula~\eqref{eq:max-fiber-factorization} becomes
\[
\sum_{U\in\mathcal F_Q^{\max}}
\beta^{\ex(U)}x^{\wt(U)}
=
x_1^3x_2^2x_3^3x_4
(1+\beta x_1)(1+\beta x_2)^2.
\]
The tableau of rank $3$ is $T$ in Example~\ref{ex:running}.
\endgroup
\end{example}

\subsection{The involution as a Boolean coordinate change}
\label{subsec:iota-max-fiber}

\begin{lemma}
\label{lem:iota-max-preservation}
For every $T\in\SVT_n(\lambda/\mu)$ and every
$u\in\lambda/\mu$,
\[
(\iota(T))_u^{\max}=T_u^{\max}.
\]
\end{lemma}

\begin{proof}
For $T=T_0$, the assertion follows from $\iota(T_0)=T_0$.
Suppose $T\ne T_0$. Only the first discrepancy is affected. Let
$u_\iota(T)=(i,j)$ and $k_0=i-\mu'_j$. Suppose first that
$k_0\notin T_{i,j}$. If the box immediately above $(i,j)$ belongs to
$\lambda/\mu$, it precedes the first discrepancy and contains the
singleton $\{k_0-1\}$. Column strictness implies
$T_{i,j}^{\min}\ge k_0$, and the absence of $k_0$ implies
$T_{i,j}^{\min}>k_0$. If $(i,j)$ is at the top of its column,
$k_0=1$. Positivity of the entries and the absence of $1$ imply the
same strict inequality. Inserting $k_0$ does not change the maximum.

Suppose that $k_0\in T_{i,j}$. The comparison with the box above,
or positivity at the top of a column, implies
$k_0=T_{i,j}^{\min}$. Since $(i,j)$ is a discrepancy, this box is not
the singleton $\{k_0\}$. Removing $k_0$ leaves its maximum unchanged.
All other boxes are fixed by $\iota$.
\end{proof}

\begin{proposition}
\label{prop:iota-coordinate}
Let $Q\in\SST_n(\lambda/\mu)$ and
$T\in\mathcal F_Q^{\max}$ with $T\ne T_0$. Write
\[
u_\iota(T)=(i,j),
\qquad
k_0=i-\mu'_j.
\]
One has
\[
(u_\iota(T),k_0)\in\Omega(Q),
\]
and toggling this coordinate produces $\iota(T)\in\mathcal F_Q^{\max}$. Exactly one
of $\Phi_Q(T)$ and $\Phi_Q(\iota(T))$ contains
$(u_\iota(T),k_0)$. The tableaux $T$ and $\iota(T)$ agree in all other
coordinates and are adjacent in the Boolean lattice $\mathcal F_Q^{\max}$.
The first discrepancy of $\iota(T)$ is again $u_\iota(T)$.
\end{proposition}

\begin{proof}
The argument in the proof of Lemma~\ref{lem:iota-max-preservation} yields
$k_0<Q_{i,j}$. If $(i,j-1)\in\lambda/\mu$, that box precedes
$u_\iota(T)$, and
\[
Q_{i,j-1}=i-\mu'_{j-1}\le k_0.
\]
If $(i-1,j)\in\lambda/\mu$, likewise $Q_{i-1,j}=k_0-1<k_0$.
These are precisely the conditions in \eqref{eq:omega-Q}, hence
$(u_\iota(T),k_0)\in\Omega(Q)$. Proposition~\ref{prop:max-fiber-boolean}
shows that toggling this coordinate produces a tableau in $\mathcal F_Q^{\max}$,
which is $\iota(T)$ by Definition~\ref{def:iota}. The proof of
Lemma~\ref{lem:iota-max-preservation} also shows that the toggled box remains
different from the prescribed singleton. All preceding boxes are unchanged,
and the first discrepancy remains $u_\iota(T)$. The remaining assertions
follow from the description in Boolean coordinates.
\end{proof}

Every fiber $\mathcal F_Q^{\max}$ is invariant under $\iota$. Since the first
discrepancy is unchanged, a second application toggles the same coordinate.
Thus $\iota^2(T)=T$ for $T\ne T_0$, and $T_0$ is the unique fixed point.
If $Q=T_0$, $\Omega(Q)=\varnothing$ and $\mathcal F_Q^{\max}=\{T_0\}$. Otherwise,
$\iota$ pairs the vertices of $\mathcal F_Q^{\max}$ along Boolean edges.

\begin{example}
\label{ex:iota-running-max-fiber}
For the tableau $T$ in Example~\ref{ex:running},
\[
\Phi_Q(T)
=
\{((1,4),1),((2,3),2),((3,2),2)\}.
\]
Its first discrepancy is $(3,2)$ and $k_0=2$, with
\[
\Phi_Q(\iota(T))
=
\Phi_Q(T)\setminus\{((3,2),2)\}.
\]
The two tableaux are adjacent vertices of ranks $3$ and $2$ in
$\mathcal F_Q^{\max}$.
\end{example}

At $\beta=-1$ and $x_1=\cdots=x_n=1$, the weighted factorization in
Corollary~\ref{cor:max-fiber-factorization} specializes to the
fiberwise cancellation in \cite[Appendix~C]{FNSOddness}.
Proposition~\ref{prop:iota-coordinate} identifies the
pairing by $\iota$ with Boolean edges.

\section{The involution and the crystal operators}
\label{sec:iota-crystal}

\begin{lemma}
\label{lem:crystal-excess}
Every nonzero application of a crystal operator preserves excess.
If $e_r(T)\ne0$,
\[
\ex(e_r(T))=\ex(T),
\]
and if $f_r(T)\ne0$,
\[
\ex(f_r(T))=\ex(T).
\]
\end{lemma}

\begin{proof}
In the replacement case, one entry is replaced by another entry in
the same box. In the transfer case, one entry is removed from an
adjacent box and one entry is inserted into the selected box. In
either case, the total number of entries is unchanged.
\end{proof}

Fix $T\ne T_0$ and write
\[
u=u_\iota(T)=(i,j),
\qquad
k_0=i-\mu'_j.
\]

For every raising color $r\ne k_0$, the involution commutes with $e_r$. For lowering colors $r\notin\{k_0-1,k_0\}$, commutation is characterized by the position of the box selected by $f_r$. The raising color $k_0$ requires a separate analysis of the reduced signature. All compositions use the conventions $\iota(0)=0$ and $e_r(0)=f_r(0)=0$.

\subsection{The position selected by a raising operator}
\label{subsec:raising-position}

The column order defining the first discrepancy imposes a restriction on every raising move.

\begin{lemma}
\label{lem:e-not-before}
If $e_r(T)\ne0$, the box containing the occurrence selected by $e_r$ does not precede $u$ in the column order.
\end{lemma}

\begin{proof}
Every box preceding $u$ is the singleton prescribed by $T_0$.
Suppose that such a box contains $r+1$. Since $r+1\ge2$, the box immediately above it belongs to the skew diagram and, because it precedes $u$, contains the entry $r$ prescribed by $T_0$. The same tableau column contains both $r$ and $r+1$ and contributes no sign to the $r$-signature. Thus the occurrence selected by $e_r$ cannot lie in a box preceding $u$.
\end{proof}

For $r\notin\{k_0-1,k_0\}$, the involution changes neither of the values that determine the $r$-signature. The neighboring color $k_0-1$ requires a separate comparison.

\begin{lemma}
\label{lem:neighboring-signature}
Assume $k_0>1$. The reduced $(k_0-1)$-signatures of $T$ and $\iota(T)$ have the same unpaired minus signs in the same order. Their numbers of unpaired plus signs differ by one.
\end{lemma}

\begin{proof}
The box immediately above $u$ belongs to the skew diagram and contains the singleton $\{k_0-1\}$. The involution changes the contribution of the tableau column containing $u$ by
\[
+\longleftrightarrow\varnothing,
\]
leaving every minus sign unchanged.

Every column strictly to the left of the column containing $u$ precedes the first discrepancy and agrees with $T_0$. Since $k_0>1$, any such column containing $k_0$ also contains $k_0-1$ and contributes no minus to the $(k_0-1)$-signature. The column containing $u$ contributes either $+$ or no sign, and every minus lies strictly to its right.

A minus is paired only with a plus to its right. Changing the contribution of the column containing $u$ leaves the pairing status of every minus unchanged, and the changed plus cannot be paired because no minus lies to its left. Hence the unpaired minus signs coincide and the numbers of unpaired plus signs differ by one.
\end{proof}

\subsection{Raising colors away from \texorpdfstring{$k_0$}{k0}}
\label{subsec:raising-away}

Lemmas~\ref{lem:e-not-before} and~\ref{lem:neighboring-signature} imply a commutation relation for every raising color except $k_0$.

\begin{theorem}
\label{thm:raising-away}
For every $1\le r<n$ with $r\ne k_0$,
\[
e_r\iota(T)=\iota e_r(T).
\]
\end{theorem}

\begin{proof}
Assume first that
\[
r\notin\{k_0-1,k_0\}.
\]
The involution changes neither $r$ nor $r+1$. The tableaux $T$ and $\iota(T)$ have identical reduced $r$-signatures. The operator $e_r$ vanishes on one tableau if and only if it vanishes on the other. If the operator is nonzero, it selects the same occurrence in both tableaux.

By Lemma~\ref{lem:e-not-before}, the selected box does not precede $u$. In the transfer case, the source box immediately to its left contains both $r$ and $r+1$ and is a multicell. Since every box preceding $u$ is a singleton, the transfer source cannot precede $u$ either. No box preceding the first discrepancy is changed by $e_r$.

The box $u$ remains a discrepancy. If $u$ is the selected box, the crystal move changes only the entries $r$ and $r+1$. Neither of these is $k_0$, and the content cannot be the singleton $\{k_0\}$. If $u$ is the transfer source, it still contains $r$ after the transfer, and $r\ne k_0$. It remains the first discrepancy in both compositions.

The insertion or removal of $k_0$ does not affect the reduced $r$-signature or the transfer condition, since $k_0\notin\{r,r+1\}$. If the two operations meet in the same box, they change distinct entries. Otherwise they act on disjoint boxes. The equality
\[
e_r\iota(T)=\iota e_r(T)
\]
follows.

The remaining color is $r=k_0-1$, which requires $k_0>1$. By Lemma~\ref{lem:neighboring-signature}, $e_r$ vanishes on $T$ if and only if it vanishes on $\iota(T)$, and otherwise it selects the same occurrence in both tableaux.

The box $u$ contributes either $+$ or no sign to the $(k_0-1)$-signature and is not selected by $e_r$. It cannot be a transfer source either: the box immediately above $u$ contains the entry $k_0-1$, and column strictness excludes $k_0-1$ from $u$. The box $u$ does not contain both $k_0-1$ and $k_0$.

The selected box does not precede $u$ by Lemma~\ref{lem:e-not-before}. A transfer source is a multicell, and every box preceding $u$ is a singleton. A transfer source cannot precede $u$. The operator $e_r$ changes no box preceding $u$ and does not change $u$ itself. The first discrepancy remains $u$. Since the selected occurrence and the transfer condition are the same for $T$ and $\iota(T)$, the two operations commute.
\end{proof}

\subsection{Lowering operators and the positional obstruction}
\label{subsec:lowering-generic}

A lowering operator can select a box preceding the first discrepancy. For $r\notin\{k_0-1,k_0\}$, this is the only obstruction to commutation. Theorem~\ref{thm:raising-away} excludes only $k_0$ for raising operators.

\begin{proposition}
\label{prop:lowering-generic}
Let $1\le r<n$ with
\[
r\notin\{k_0-1,k_0\}.
\]
One has
\[
f_r\iota(T)=\iota f_r(T)
\]
if and only if either $f_r(T)=0$ or the box containing the occurrence selected by $f_r$ does not precede $u$ in the column order.
\end{proposition}

\begin{proof}
Since $k_0\notin\{r,r+1\}$, the tableaux $T$ and $\iota(T)$ have identical reduced $r$-signatures. In particular,
\[
f_r(T)=0
\quad\Longleftrightarrow\quad
f_r(\iota(T))=0.
\]
If both sides vanish, the required equality follows from $\iota(0)=0$.

Assume $f_r(T)\ne0$, and let $v$ be the box containing the occurrence selected by $f_r$. The same box is selected in $\iota(T)$.

Suppose first that $v$ precedes $u$. Since every box preceding the first discrepancy agrees with $T_0$, one has
\[
T_v=\{r\}.
\]
Whether $f_r$ acts by replacement or by transfer, the content of $v$ changes and no longer agrees with $T_0$. Every box preceding $v$ remains unchanged. The box $v$ is the first discrepancy of $f_r(T)$.

The involution applied to $f_r(T)$ toggles the entry $r$ prescribed by $T_0$ in the box $v$. Applying $\iota$ to $T$ changes only the entry $k_0$ at $u$. Since $k_0\notin\{r,r+1\}$, this does not change the reduced $r$-signature or the transfer condition for $f_r$.
\[
(f_r\iota(T))_v=(f_r(T))_v,\qquad
(\iota f_r(T))_v\ne(f_r(T))_v.
\]
The two compositions are unequal:
\[
f_r\iota(T)\ne\iota f_r(T).
\]

Suppose next that $v$ does not precede $u$. If $v$ follows $u$, the selected box and, in the transfer case, the source box immediately to its right both follow $u$. The operator $f_r$ changes no box at or before the first discrepancy. The first discrepancy remains $u$, and the changes made by $f_r$ and $\iota$ occur in disjoint boxes. The two operations commute.

The case $v=u$ remains. In the replacement case, $f_r$ replaces $r$ by $r+1$ in $u$. Since neither $r$ nor $r+1$ is $k_0$, the box cannot become the singleton $\{k_0\}$. In the transfer case, $f_r$ retains the selected entry $r$ in $u$ and inserts $r+1$ there. The box again remains different from $\{k_0\}$. The box $u$ remains the first discrepancy.

The transfer condition for $f_r$ depends on the box immediately to the right of $u$, which is unchanged by $\iota$. The insertion or removal of $k_0$ commutes with the local change involving $r$ and $r+1$. The equality
\[
f_r\iota(T)=\iota f_r(T)
\]
follows.
\end{proof}

\begin{example}
\label{ex:lowering-obstruction}
Let $\lambda=(2,1)$, $\mu=\varnothing$, and $n=3$. For
\[
T=
\begin{ytableau}
1&3\\
2
\end{ytableau}
\]
the first discrepancy is $u=(1,2)$ and $k_0=1$. The operator $f_2$ selects the entry $2$ in $(2,1)$, which precedes $u$. The two compositions are
\[
f_2\iota(T)=
\begin{ytableau}
1&13\\
3
\end{ytableau}
\neq
\begin{ytableau}
1&3\\
23
\end{ytableau}
=\iota f_2(T).
\]
The selected box violates the positional condition in Proposition~\ref{prop:lowering-generic}.
\end{example}

The positional criterion in Proposition~\ref{prop:lowering-generic} does not extend unchanged to the neighboring color $k_0-1$. For $\lambda=(1,1)$, $\mu=\varnothing$, $n=3$, and
\[
T=
\begin{ytableau}
1\\
23
\end{ytableau},
\]
one has $u=(2,1)$ and $k_0=2$. Since the tableau column contains both $1$ and $2$, $f_1(T)=0$. On the other hand,
\[
\iota(T)=
\begin{ytableau}
1\\
3
\end{ytableau},
\qquad
f_1\iota(T)=
\begin{ytableau}
2\\
3
\end{ytableau}
\ne 0=\iota f_1(T).
\]

\subsection{The exceptional raising color}
\label{subsec:exceptional-color}

At the first discrepancy one has $k_0<n$. The argument in Lemma~\ref{lem:iota-max-preservation} shows that every entry of $T_u$ is at least $k_0$. If $k_0\notin T_u$, $T_u$ contains an entry larger than $k_0$. If $k_0\in T_u$, the inequality $T_u\ne\{k_0\}$ implies that it contains another entry, necessarily larger than $k_0$. The operator $e_{k_0}$ is defined.

By the construction of $\iota$, the tableaux in every nontrivial two-element orbit have the same first discrepancy. Among the two tableaux, one satisfies
\begin{equation}
\label{eq:k0-absent}
k_0\notin T_u.
\end{equation}
Since $\iota^2=\mathrm{id}$ and $\iota(0)=0$, the commutation relation for one tableau in the orbit is equivalent to that for the other. Assume \eqref{eq:k0-absent} throughout this subsection. If $k_0+1\notin T_u$ or the minus contributed by the column containing $u$ is paired, Lemma~\ref{lem:exceptional-easy} applies. The remaining case is treated in Theorem~\ref{thm:exceptional}.

\begin{lemma}
\label{lem:exceptional-easy}
Under \eqref{eq:k0-absent}, the equality
\[
e_{k_0}\iota(T)=\iota e_{k_0}(T)
\]
holds in either of these cases:
\begin{enumerate}
\item $k_0+1\notin T_u$.
\item $k_0+1\in T_u$, and the minus contributed by the column containing $u$ is paired in the $k_0$-signature.
\end{enumerate}
\end{lemma}

\begin{proof}
Suppose first that $k_0+1\notin T_u$. Under \eqref{eq:k0-absent}, the box $u$ contains neither $k_0$ nor $k_0+1$. Its minimum is at least $k_0+2$. The boxes above $u$ agree with $T_0$. Column strictness excludes $k_0$ and $k_0+1$ from the boxes below $u$. The column containing $u$ contributes no sign to the $k_0$-signature of $T$ and contributes $+$ to that of $\iota(T)$.

Every column strictly to the left agrees with $T_0$ and contributes no minus to the $k_0$-signature. The additional plus does not change any pairing. The same occurrence is selected in both tableaux, or the operator vanishes on both. In the nonzero case, the selected box and any transfer source lie after $u$. The box $u$ cannot be a transfer source because it does not contain $k_0+1$. The first discrepancy remains $u$, and the two operations commute.

Suppose next that $k_0+1\in T_u$ and that the minus contributed by the column containing $u$ is paired. Inserting $k_0$ removes this minus from the signature. The pairings of all minuses strictly to its right are unchanged, since these minuses are processed earlier. No minus lies to its left, since those columns agree with $T_0$. The reduced signatures have the same unpaired minuses in the same order. If there are no unpaired minuses, both compositions are $0$.

In the nonzero case, the same occurrence is selected strictly to the right of $u$ in both tableaux. The plus paired with the minus from the column containing $u$ lies strictly to the left of every unpaired minus. Otherwise an unpaired minus has an available plus at its pairing step, a contradiction. This plus separates the column containing $u$ from the selected column. The selected box is not immediately to the right of $u$, and $u$ cannot become a transfer source after inserting $k_0$. The selected box and any transfer source lie after $u$ and are unchanged by $\iota$. The first discrepancy remains $u$, and the two operations commute.
\end{proof}

Assume that $k_0+1\in T_u$ and that the minus contributed by its column survives in the reduced $k_0$-signature. No column strictly to the left contributes a minus, since every such column agrees with $T_0$. This minus is the leftmost unpaired minus, and $e_{k_0}$ selects the occurrence of $k_0+1$ in $u$.

The box immediately to the left of $u$, if present, precedes the first discrepancy and is a singleton. It cannot contain both $k_0$ and $k_0+1$. The operator $e_{k_0}$ acts on $T$ by replacing the selected $k_0+1$ by $k_0$. In particular, $(e_{k_0}(T))_u=\{k_0\}$ if and only if $T_u=\{k_0+1\}$.

\begin{theorem}
\label{thm:exceptional}
Assume \eqref{eq:k0-absent}, and suppose that $e_{k_0}$ selects the occurrence of $k_0+1$ in $u=(i,j)$. Put $v=(i,j+1)$. The equality
\[
e_{k_0}\iota(T)=\iota e_{k_0}(T)
\]
holds if and only if conditions \textup{(i)} and \textup{(ii)} hold.
\begin{enumerate}
\item One has $T_u=\{k_0+1\}$, the box $v$ belongs to $\lambda/\mu$, and $v$ is the first discrepancy of $e_{k_0}(T)$.
\item In $\iota(T)$, the operator $e_{k_0}$ selects the occurrence of $k_0+1$ in $v$.
\end{enumerate}
\end{theorem}

\begin{proof}
The box immediately to the left of $u$, if present, is a singleton. Since $e_{k_0}$ selects $u$, it acts on $T$ by replacement.

Assume first that
\[
e_{k_0}\iota(T)=\iota e_{k_0}(T).
\]
Since $e_{k_0}(T)\ne0$, one has $\iota e_{k_0}(T)\ne0$. The assumed equality implies $e_{k_0}\iota(T)\ne0$. Under \eqref{eq:k0-absent}, the involution inserts $k_0$, and
\[
\ex(\iota(T))=\ex(T)+1.
\]
Lemma~\ref{lem:crystal-excess} implies
\[
\ex(e_{k_0}\iota(T))=\ex(T)+1,
\qquad
\ex(e_{k_0}(T))=\ex(T).
\]

Suppose that $u$ remains the first discrepancy of $e_{k_0}(T)$. The replacement has inserted $k_0$ in this box. Since the box is still a discrepancy, it contains another entry as well. The involution removes $k_0$ from this box, and
\[
\ex(\iota e_{k_0}(T))=\ex(T)-1,
\]
a contradiction. The box $u$ agrees with $T_0$ after applying $e_{k_0}$, and
\[
(e_{k_0}(T))_u=\{k_0\},
\qquad
T_u=\{k_0+1\}.
\]

In $\iota(T)$, the box $u$ is $\{k_0,k_0+1\}$ and its column contributes no sign. Equality of the two compositions requires $e_{k_0}\iota(T)$ to remove the occurrence $k_0+1$ from $u$. Since $u$ itself is not selected, this occurs only if $u$ is the transfer source for a selected box immediately to its right. The box $v=(i,j+1)$ belongs to the skew diagram, and $e_{k_0}$ selects the occurrence of $k_0+1$ in $v$ from $\iota(T)$.

On the other side, $\iota e_{k_0}(T)$ changes the first discrepancy of $e_{k_0}(T)$. Equality forces that discrepancy to be $v$. Conditions \textup{(i)} and \textup{(ii)} follow.

Conversely, assume conditions \textup{(i)} and \textup{(ii)}. First,
\[
\mu'_j=\mu'_{j+1}.
\]
Suppose instead that $\mu'_j>\mu'_{j+1}$. Since $k_0=i-\mu'_j$, the column containing $v$ contains, strictly above $v$, the box whose entry in $T_0$ is $k_0$. By condition \textup{(i)}, $v$ is the first discrepancy of $e_{k_0}(T)$, and this box above $v$ agrees with $T_0$. The operator $e_{k_0}$ changes only $u$, and this occurrence of $k_0$ is present in the column containing $v$ in both $T$ and $\iota(T)$.

Condition \textup{(ii)} places an occurrence of $k_0+1$ in the column containing $v$. That column contains both $k_0$ and $k_0+1$ and contributes no sign to the $k_0$-signature, contrary to condition \textup{(ii)}. Thus $\mu'_j=\mu'_{j+1}$, and the entry of $T_0$ in $v$ is $k_0$.

Since the column containing $v$ contributes a minus under condition \textup{(ii)}, it contains no $k_0$. The first discrepancy $v$ of $e_{k_0}(T)$ receives $k_0$ under $\iota$. For $e_{k_0}\iota(T)$, the box $u$ contains $\{k_0,k_0+1\}$ and condition \textup{(ii)} selects the occurrence $k_0+1$ in $v$. The transfer removes $k_0+1$ from $u$ and inserts $k_0$ into $v$. These are the changes made by $\iota e_{k_0}(T)$. The two tableaux are equal.
\end{proof}

The two conditions in Theorem~\ref{thm:exceptional} have the following interpretation. Condition~\textup{(i)} places the first discrepancy of $e_{k_0}(T)$ immediately to the right of $u$. Condition~\textup{(ii)} requires the column containing $v$ to contribute the leftmost unpaired minus in the reduced $k_0$-signature of $\iota(T)$. The contents of $u$ and $v$ alone do not specify that reduced signature.

\begin{example}
\label{ex:exceptional-square}
Let $\lambda=(4,2)$, $\mu=(1)$, and $n=2$. For
\[
T=
\begin{ytableau}
\none&1&2&2\\
1&2
\end{ytableau},
\]
the first discrepancy is $u=(1,3)$ and $k_0=1$. The unpaired minuses of $T$ are in columns $3$ and $4$. The operator $e_1$ replaces the entry $2$ at $u$ by $1$, and the first discrepancy of $e_1(T)$ is $v=(1,4)$. This verifies condition~\textup{(i)} of Theorem~\ref{thm:exceptional}. In $\iota(T)$, column $3$ contains both $1$ and $2$, and column $4$ contributes the only unpaired minus. The selected occurrence lies in $v$, as required by condition~\textup{(ii)}. Figure~\ref{fig:exceptional-square} displays the replacement and transfer.

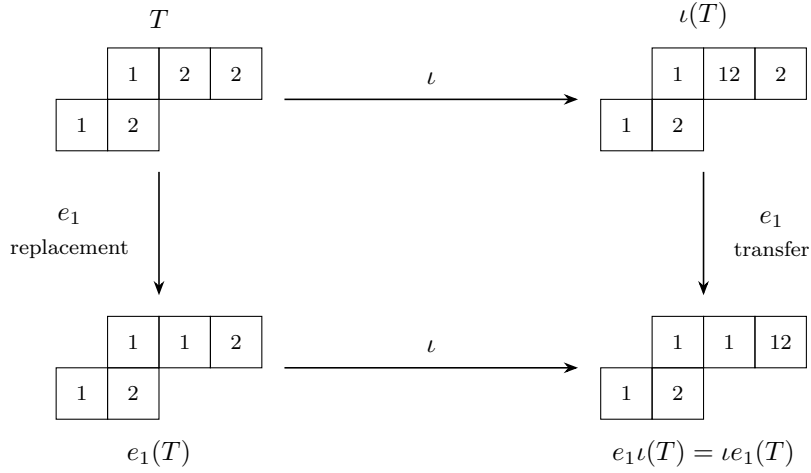
\begin{figure}[!t]
\centering
\begin{tikzpicture}[x=0.68cm,y=0.68cm,>=Stealth]
  \tikzset{
    maparrow/.style={-{Stealth[length=1.9mm]},line width=0.60pt}
  }

  \begin{scope}[shift={(0,5.22)}]
    \foreach \x/\y/\t in {2/2/1,3/2/2,4/2/2,1/1/1,2/1/2}{
      \draw (\x-0.5,\y-0.5) rectangle (\x+0.5,\y+0.5);
      \node at (\x,\y) {\scriptsize $\t$};
    }
    \node[above] at (2.5,2.68) {\small $T$};
  \end{scope}

  \begin{scope}[shift={(10.59,5.22)}]
    \foreach \x/\y/\t in {2/2/1,3/2/12,4/2/2,1/1/1,2/1/2}{
      \draw (\x-0.5,\y-0.5) rectangle (\x+0.5,\y+0.5);
      \node at (\x,\y) {\scriptsize $\t$};
    }
    \node[above] at (2.5,2.68) {\small $\iota(T)$};
  \end{scope}

  \begin{scope}[shift={(0,0)}]
    \foreach \x/\y/\t in {2/2/1,3/2/1,4/2/2,1/1/1,2/1/2}{
      \draw (\x-0.5,\y-0.5) rectangle (\x+0.5,\y+0.5);
      \node at (\x,\y) {\scriptsize $\t$};
    }
    \node[below] at (2.5,0.30) {\small $e_1(T)$};
  \end{scope}

  \begin{scope}[shift={(10.59,0)}]
    \foreach \x/\y/\t in {2/2/1,3/2/1,4/2/12,1/1/1,2/1/2}{
      \draw (\x-0.5,\y-0.5) rectangle (\x+0.5,\y+0.5);
      \node at (\x,\y) {\scriptsize $\t$};
    }
    \node[below] at (2.5,0.30) {\small $e_1\iota(T)=\iota e_1(T)$};
  \end{scope}

  \draw[maparrow] (4.94,6.72)--(10.65,6.72)
    node[midway,above=2pt] {\small $\iota$};
  \draw[maparrow] (4.94,1.50)--(10.65,1.50)
    node[midway,above=2pt] {\small $\iota$};

  \draw[maparrow] (2.50,5.31)--(2.50,2.91)
    node[midway,left=2.6mm,align=center] {\small $e_1$\\[-1pt]\scriptsize replacement};
  \draw[maparrow] (13.09,5.31)--(13.09,2.91)
    node[midway,right=2.6mm,align=center] {\small $e_1$\\[-1pt]\scriptsize transfer};
\end{tikzpicture}
\caption{The commuting square for the exceptional color $k_0=1$. The left crystal edge is a replacement and the right crystal edge is a transfer.}
\label{fig:exceptional-square}
\end{figure}
\end{example}

\FloatBarrier

\section{Rotation and the dual involution}
\label{sec:duality}

Rotation with alphabet reversal exchanges minimum and maximum entries. The Boolean structure for fibers determined by minima and the dual involution are obtained by applying this rotation to the corresponding constructions for fibers determined by maxima. Related dualities among tableau models for Grothendieck polynomials are developed in \cite{Hawkes2024}.

\subsection{Rotation and fibers determined by minima}
\label{subsec:rotation}

Fix and retain the $\ell(\lambda)\times\lambda_1$ rectangle containing $\lambda/\mu$. Pad every partition in this rectangle by zeros to length $\ell(\lambda)$.

\begin{definition}
\label{def:rotation}
For a partition $\alpha$ in the fixed rectangle, define
\[
\alpha_i^\vee
=
\lambda_1-\alpha_{\ell(\lambda)+1-i}.
\]
The skew diagram $\mu^\vee/\lambda^\vee$ is the $180^\circ$ rotation of $\lambda/\mu$. For $u=(i,j)\in\lambda/\mu$, write
\[
u^\vee=(\ell(\lambda)+1-i,\lambda_1+1-j).
\]
For $T\in\SVT_n(\lambda/\mu)$, define $\rho(T)$ on $\mu^\vee/\lambda^\vee$ by
\[
(\rho(T))_{u^\vee}
=
\{n+1-k:k\in T_u\}.
\]
\end{definition}

\begin{proposition}
\label{prop:rotation}
The map
\[
\rho:\SVT_n(\lambda/\mu)
\longrightarrow
\SVT_n(\mu^\vee/\lambda^\vee)
\]
is a bijection and preserves excess. It exchanges minimum and maximum entries according to
\begin{align}
(\rho(T))_{u^\vee}^{\min}
&=n+1-T_u^{\max},
\label{eq:rho-min}\\
(\rho(T))_{u^\vee}^{\max}
&=n+1-T_u^{\min}.
\label{eq:rho-max}
\end{align}
Applying the same construction twice in the fixed rectangle returns $T$.
\end{proposition}

\begin{proof}
Rotation reverses the order of the boxes in every row and column, and alphabet reversal reverses numerical inequalities. Hence the weak row inequalities and strict column inequalities are preserved. Reversal of each finite set proves \eqref{eq:rho-min} and \eqref{eq:rho-max}, with every box retaining its cardinality and excess. A second application of the same rotation and alphabet reversal in the fixed rectangle returns the original tableau.
\end{proof}

\begin{example}
\label{ex:rotation}
Let $\lambda=(4,4,2)$, $\mu=(2,1)$, and $n=5$. The fixed ambient rectangle is $3\times4$, and
\[
\lambda^\vee=(2),
\qquad
\mu^\vee=(4,3,2).
\]
The shaded boxes in the diagram represent the skew diagrams $\lambda/\mu$ and $\mu^\vee/\lambda^\vee$.

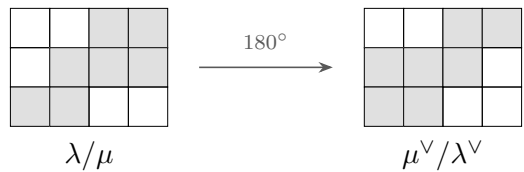
\begin{center}
\begin{tikzpicture}[x=0.52cm,y=-0.52cm,line width=0.35pt,>=Stealth]
\definecolor{rotationfill}{gray}{0.88}
\definecolor{rotationline}{gray}{0.35}
\tikzset{rotarrow/.style={-{Stealth[length=1.9mm]},line width=0.60pt,draw=rotationline}}
\begin{scope}[shift={(0,0)}]
  \foreach \row/\first/\last in {0/2/3,1/1/3,2/0/1}{
    \foreach \col in {0,...,3}{
      \ifnum\col<\first\else
        \ifnum\col>\last\else
          \fill[fill=rotationfill] (\col,\row) rectangle +(1,1);
        \fi
      \fi
      \draw (\col,\row) rectangle +(1,1);
    }
  }
  \node at (2,3.75) {$\lambda/\mu$};
\end{scope}
\begin{scope}[shift={(9.0,0)}]
  \foreach \row/\first/\last in {0/2/3,1/0/2,2/0/1}{
    \foreach \col in {0,...,3}{
      \ifnum\col<\first\else
        \ifnum\col>\last\else
          \fill[fill=rotationfill] (\col,\row) rectangle +(1,1);
        \fi
      \fi
      \draw (\col,\row) rectangle +(1,1);
    }
  }
  \node at (2,3.75) {$\mu^\vee/\lambda^\vee$};
\end{scope}
\draw[rotarrow] (4.8,1.5)--(8.2,1.5)
  node[midway,above=3pt,text=rotationline] {\scriptsize $180^\circ$};
\end{tikzpicture}
\end{center}

For
\[
T=
\begin{ytableau}
\none&\none&1&2\\
\none&23&4&5\\
4&5
\end{ytableau},
\]
the transformation $T\mapsto\rho(T)$ decomposes as
\begin{center}
\begingroup
\ytableausetup{mathmode,boxsize=1.35em,centertableaux}
\begin{tikzpicture}[>=Stealth,every node/.style={inner sep=1pt}]
\tikzset{rotarrow/.style={-{Stealth[length=1.9mm]},line width=0.60pt}}
\node (A) at (0,0) {$\displaystyle
\begin{ytableau}
\none&\none&1&2\\
\none&23&4&5\\
4&5
\end{ytableau}$};
\node[above=2.2mm of A] {\small $T$};
\node (B) at (5.7,0) {$\displaystyle
\begin{ytableau}
\none&\none&5&4\\
5&4&23\\
2&1
\end{ytableau}$};
\node[above=2.2mm of B] {\small $180^\circ$ rotation};
\node (C) at (11.4,0) {$\displaystyle
\begin{ytableau}
\none&\none&1&2\\
1&2&34\\
4&5
\end{ytableau}$};
\node[above=2.2mm of C] {\small $\rho(T)$};
\draw[rotarrow] ([xshift=2.7mm]A.east)--node[midway,above=2pt] {\small $180^\circ$}([xshift=-2.7mm]B.west);
\draw[rotarrow] ([xshift=2.7mm]B.east)--node[midway,above=2pt] {\small $k\mapsto6-k$}([xshift=-2.7mm]C.west);
\end{tikzpicture}
\endgroup
\end{center}
\end{example}

For $P\in\SST_n(\lambda/\mu)$, put
\[
\mathcal F_P^{\min}
=
\{T\in\SVT_n(\lambda/\mu):T_u^{\min}=P_u\text{ for every }u\in\lambda/\mu\}.
\]
Order $\mathcal F_P^{\min}$ by inclusion in each box: $T\le U$ if $T_u\subseteq U_u$ for every $u\in\lambda/\mu$.
Define
\[
\begin{split}
\Omega^*(P)=\{((i,j),k):{}&(i,j)\in\lambda/\mu,\ P_{i,j}<k\le n,\\
&k\le P_{i,j+1}\text{ if }(i,j+1)\in\lambda/\mu,\\
&k<P_{i+1,j}\text{ if }(i+1,j)\in\lambda/\mu\}.
\end{split}
\]

Regard $Q=\rho(P)$ as a singleton tableau on the rotated skew shape. By~\eqref{eq:rho-max}, $\rho$ identifies $\mathcal F_P^{\min}$ with the maximum fiber $\mathcal F_Q^{\max}$. The correspondence
\begin{equation}
\label{eq:rotation-boolean-coordinate}
(u,k)
\longleftrightarrow
(u^\vee,n+1-k)
\end{equation}
identifies $\Omega^*(P)$ with $\Omega(Q)$. The left neighbor of $u^\vee$ corresponds to the right neighbor of $u$, and the upper neighbor corresponds to the lower neighbor. The inequalities defining the two coordinate sets are transformed into one another by $k\mapsto n+1-k$.

\begin{corollary}
\label{cor:min-fiber}
The map
\[
\Phi_P^*:\mathcal F_P^{\min}\longrightarrow2^{\Omega^*(P)},
\qquad
\Phi_P^*(T)=\{(u,k)\in\Omega^*(P):k\in T_u\},
\]
is an isomorphism of graded posets. Under this isomorphism,
\[
\operatorname{rank}(T)=\ex(T),
\qquad
|\mathcal F_P^{\min}|=2^{|\Omega^*(P)|}.
\]
\begin{equation}
\label{eq:min-fiber-weight}
\sum_{T\in\mathcal F_P^{\min}}
\beta^{\ex(T)}x^{\wt(T)}
=
x^{\wt(P)}
\prod_{(u,k)\in\Omega^*(P)}(1+\beta x_k).
\end{equation}
\end{corollary}

\begin{proof}
Set $Q=\rho(P)$. Proposition~\ref{prop:rotation} is a bijection from $\mathcal F_P^{\min}$ to the maximum fiber $\mathcal F_Q^{\max}$ on the rotated skew shape that preserves excess, and \eqref{eq:rotation-boolean-coordinate} identifies the Boolean coordinates. Both $\rho$ and $\rho^{-1}$ preserve inclusion in each box. The poset, rank, and cardinality statements follow from Proposition~\ref{prop:max-fiber-boolean}.

For the weighted identity, let $\sigma$ be the involution of $\mathbb Z[\beta,x_1,\ldots,x_n]$ defined by
\[
\sigma(\beta)=\beta,
\qquad
\sigma(x_k)=x_{n+1-k}.
\]
For every $T$,
\[
x^{\wt(\rho(T))}=\sigma\bigl(x^{\wt(T)}\bigr).
\]
Apply $\sigma$ to the factorization \eqref{eq:max-fiber-factorization} for the maximum fiber of $Q$. The coordinate correspondence \eqref{eq:rotation-boolean-coordinate} yields \eqref{eq:min-fiber-weight}.
\end{proof}

For straight shapes, the decomposition by fixed minimum entries appears in \cite[Proposition~8.11]{MoralesZhu2022}. Corollary~\ref{cor:min-fiber} records the Boolean order, excess grading, and monomial weights for skew shapes.

\begin{example}
\label{ex:min-fiber}
For $\lambda=(2,1)$, $\mu=\varnothing$, $n=5$, and
\[
P=\begin{ytableau}1&3\\4\end{ytableau},
\]
one has
\[
\Omega^*(P)=
\{((1,1),2),((1,1),3),((1,2),4),((1,2),5),((2,1),5)\}.
\]
The fiber has cardinality $|\mathcal F_P^{\min}|=32$, and its weighted sum is
\[
\sum_{T\in\mathcal F_P^{\min}}\beta^{\ex(T)}x^{\wt(T)}
=
x_1x_3x_4(1+\beta x_2)(1+\beta x_3)(1+\beta x_4)(1+\beta x_5)^2.
\]
\end{example}

\subsection{The dual involution}
\label{subsec:iota-star}
\begingroup
\setlength{\abovedisplayskip}{7pt}
\setlength{\belowdisplayskip}{7pt}
\setlength{\abovedisplayshortskip}{5pt}
\setlength{\belowdisplayshortskip}{5pt}

For $u=(i,j)\in\lambda/\mu$, define
\[
d^*(u)=n-\lambda'_j+i.
\]
Let $T_0^*$ be the singleton tableau
\[
(T_0^*)_u=\{d^*(u)\}.
\]
Under $\rho$, $T_0^*$ is sent to the distinguished tableau $T_0$ on the rotated skew shape.

\begin{definition}
\label{def:iota-star}
Let $\iota$ on the rotated skew shape be the involution of Definition~\ref{def:iota}. Define
\begin{equation}
\label{eq:iota-star-conjugation}
\iota^*=\rho^{-1}\circ\iota\circ\rho,
\end{equation}
and set $\iota^*(0)=0$.
\end{definition}

\begin{proposition}
\label{prop:iota-star}
The map $\iota^*$ is an involution with unique fixed point $T_0^*$. For every $T\ne T_0^*$, read the boxes from right to left by columns and, within each column, from bottom to top. Let $u=u_{\iota^*}(T)$ be the first box satisfying
\[
T_u\ne(T_0^*)_u.
\]
The tableau $\iota^*(T)$ differs from $T$ only in $u$, where it inserts or removes $d^*(u)$. Its minimum entries satisfy
\[
(\iota^*(T))_v^{\min}=T_v^{\min}
\qquad(v\in\lambda/\mu).
\]
Every nontrivial application changes excess by one.
\end{proposition}

\begin{proof}
Under $u\mapsto u^\vee$, the column order on the rotated skew diagram becomes the reverse column order on $\lambda/\mu$. Alphabet reversal sends the prescribed entry of $T_0$ at $u^\vee$ to $d^*(u)$. The stated rule follows from \eqref{eq:iota-star-conjugation}. The unique fixed point follows from conjugation, and the excess statement follows because $\rho$ preserves excess. The involution $\iota$ preserves maximum entries, and \eqref{eq:rho-min}--\eqref{eq:rho-max} convert this into preservation of minimum entries under $\iota^*$.
\end{proof}

\begin{proposition}
\label{prop:iota-star-coordinate}
Let $P\in\SST_n(\lambda/\mu)$ and $T\in\mathcal F_P^{\min}$ with $T\ne T_0^*$. If $u=u_{\iota^*}(T)$,
\[
(u,d^*(u))\in\Omega^*(P),
\]
and exactly one of
$\Phi_P^*(T)$ and $\Phi_P^*(\iota^*(T))$ contains
$(u,d^*(u))$. The tableaux $T$ and $\iota^*(T)$ agree in all other
coordinates and are adjacent in the Boolean lattice $\mathcal F_P^{\min}$. If $P\ne T_0^*$, the involution has no fixed point on $\mathcal F_P^{\min}$.
\end{proposition}

\begin{proof}
Apply rotation to Proposition~\ref{prop:iota-coordinate}. Equivalently, Proposition~\ref{prop:iota-star} shows that $\iota^*$ preserves the minimum tableau and changes only the occurrence of $d^*(u)$ in the box $u$. Corollary~\ref{cor:min-fiber} identifies such optional entries with the coordinates of $\Omega^*(P)$.
\end{proof}

\begin{example}
\label{ex:rotation-duality}
For the tableau $T$ in Example~\ref{ex:running}, the rotated skew shape is again $(5,4,3)/(2,1)$. The first discrepancy for $\iota^*$ is $u=(1,5)$, and
\[
d^*(u)=4.
\]
The involution $\iota^*$ inserts $4$ in this box. Figure~\ref{fig:rotation-duality} displays the conjugation relation.

\end{example}

\endgroup
\subsection{Crystal operators under rotation}
\label{subsec:rotation-crystal}

\begin{lemma}
\label{lem:rotation-crystal}
For $1\le r<n$,
\[
\rho(e_r(T))=f_{n-r}(\rho(T)),
\qquad
\rho(f_r(T))=e_{n-r}(\rho(T)),
\]
with $\rho(0)=0$.
\end{lemma}

\begin{proof}
Alphabet reversal sends the pair $r,r+1$ to $n-r+1,n-r$. Rotation reverses the column order. The signs in the reduced signature are interchanged and their order is reversed. The leftmost unpaired minus selected by $e_r$ becomes the rightmost unpaired plus selected by $f_{n-r}$. The replacement rule is reversed, and rotation interchanges the left and right neighbors in the transfer rule.
\end{proof}

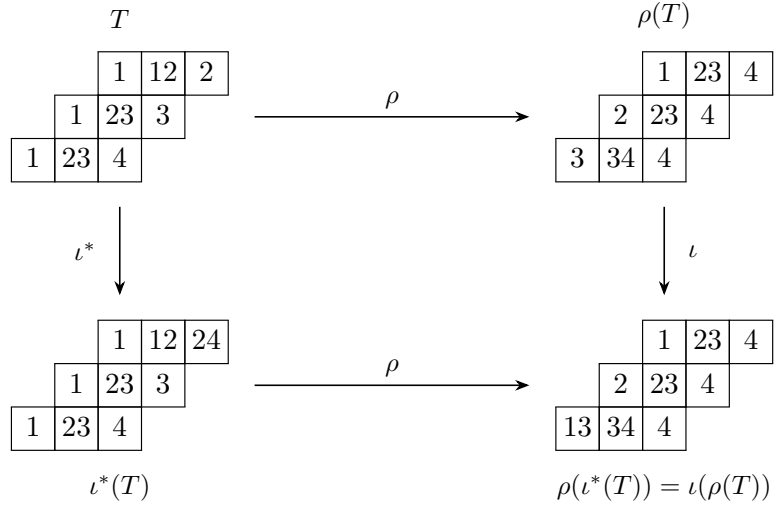
\begin{figure}[H]
\centering
\begingroup
\ytableausetup{mathmode,boxsize=1.45em,centertableaux}
\begin{tikzpicture}[>=Stealth,every node/.style={inner sep=1pt}]
\tikzset{maparrow/.style={-{Stealth[length=1.9mm]},line width=0.60pt}}
\node (A) at (0,3.55) {$\displaystyle
\begin{ytableau}
\none&\none&1&12&2\\
\none&1&23&3\\
1&23&4
\end{ytableau}$};
\node[above=2.2mm of A] {\small $T$};
\node (B) at (7.20,3.55) {$\displaystyle
\begin{ytableau}
\none&\none&1&23&4\\
\none&2&23&4\\
3&34&4
\end{ytableau}$};
\node[above=2.2mm of B] {\small $\rho(T)$};
\node (C) at (0,0) {$\displaystyle
\begin{ytableau}
\none&\none&1&12&24\\
\none&1&23&3\\
1&23&4
\end{ytableau}$};
\node[below=2.2mm of C] {\small $\iota^*(T)$};
\node (D) at (7.20,0) {$\displaystyle
\begin{ytableau}
\none&\none&1&23&4\\
\none&2&23&4\\
13&34&4
\end{ytableau}$};
\node[below=2.2mm of D,align=center] {\small $\rho(\iota^*(T))=\iota(\rho(T))$};
\draw[maparrow] ([xshift=3mm]A.east)--node[midway,above=2pt] {\small $\rho$}([xshift=-3mm]B.west);
\draw[maparrow] ([xshift=3mm]C.east)--node[midway,above=2pt] {\small $\rho$}([xshift=-3mm]D.west);
\draw[maparrow] ([yshift=-2.8mm]A.south)--node[midway,left=2.7mm] {\small $\iota^*$}([yshift=2.8mm]C.north);
\draw[maparrow] ([yshift=-2.8mm]B.south)--node[midway,right=2.7mm] {\small $\iota$}([yshift=2.8mm]D.north);
\end{tikzpicture}
\endgroup
\caption{Rotation with alphabet reversal conjugates $\iota^*$ to $\iota$ in Example~\ref{ex:rotation-duality}.}
\label{fig:rotation-duality}
\end{figure}

Let $T\ne T_0^*$, let $u=u_{\iota^*}(T)=(i,j)$ be its first discrepancy in reverse column order, and write
\[
d^*=d^*(u).
\]
Since $i\le\lambda_j'$, one has $d^*\le n$. Suppose $d^*=1$. If $i<\lambda_j'$, the box immediately below $u$ belongs to $\lambda/\mu$, precedes $u$ in reverse column order, and contains the prescribed singleton $\{2\}$. Column strictness forces $T_u=\{1\}$, contrary to the choice of $u$. If $i=\lambda_j'$, one has $d^*=n$. The assumption $d^*=1$ forces $n=1$, and the alphabet again forces $T_u=\{1\}$. Thus $2\le d^*\le n$.

\begin{corollary}
\label{cor:iota-star-lowering}
For every $1\le r<n$ with $r\ne d^*-1$,
\[
f_r\iota^*(T)=\iota^*f_r(T).
\]
\end{corollary}

\begin{proof}
Apply Theorem~\ref{thm:raising-away} to $\rho(T)$ and use Lemma~\ref{lem:rotation-crystal}. The entry changed by $\iota$ on the rotated tableau is $n+1-d^*$, and
\[
n-r\ne n+1-d^*
\quad\Longleftrightarrow\quad
r\ne d^*-1.
\]
\end{proof}

\begin{corollary}
\label{cor:iota-star-raising}
For $1\le r<n$ with $r\notin\{d^*-1,d^*\}$,
\[
e_r\iota^*(T)=\iota^*e_r(T)
\]
if and only if either $e_r(T)=0$ or the box selected by $e_r$ does not precede $u$ in the reverse column order.
\end{corollary}

\begin{proof}
Apply Proposition~\ref{prop:lowering-generic} to $\rho(T)$ and use Lemma~\ref{lem:rotation-crystal}. Rotation converts the column order on the rotated diagram into the reverse column order on the original diagram.
\end{proof}

The exceptional lowering color $d^*-1$ is the rotated form of the exceptional raising color in Theorem~\ref{thm:exceptional}. Orient each two-element $\iota^*$-orbit by requiring
\begin{equation}
\label{eq:dstar-absent}
d^*\notin T_u.
\end{equation}
If $d^*-1\notin T_u$, or if the sign corresponding to the occurrence of $d^*-1$ in the column of $u$ is canceled during the reduction of the $(d^*-1)$-signature,
\[
f_{d^*-1}\iota^*(T)=\iota^*f_{d^*-1}(T).
\]
The remaining case has $d^*-1\in T_u$, with $f_{d^*-1}$ selecting this occurrence. Put $v=(i,j-1)$ for $u=(i,j)$.

\begin{corollary}
\label{cor:iota-star-exceptional}
Under \eqref{eq:dstar-absent} and the preceding selection assumption,
\[
f_{d^*-1}\iota^*(T)=\iota^*f_{d^*-1}(T)
\]
holds if and only if conditions \textup{(i)} and \textup{(ii)} hold.
\begin{enumerate}
\item The box $v$ belongs to $\lambda/\mu$, the tableau $f_{d^*-1}(T)$ satisfies
\[
(f_{d^*-1}(T))_u=\{d^*\},
\]
and $v$ is its first discrepancy in reverse column order.
\item In $\iota^*(T)$, the box $v$ contains $d^*-1$, and the sign from this occurrence is the rightmost unpaired plus in the reduced $(d^*-1)$-signature.
\end{enumerate}
\end{corollary}

\begin{proof}
Apply Theorem~\ref{thm:exceptional} to $\rho(T)$. The box immediately to the right of the discrepancy on the rotated diagram becomes the box immediately to the left of $u$. Alphabet reversal sends the selected value to $d^*-1$, and signature reversal sends the leftmost unpaired minus to the rightmost unpaired plus.
\end{proof}

Every nonzero ordinary crystal edge preserves tableau excess. A nontrivial application of $\iota$ or $\iota^*$ changes excess by one. Square-root crystal structures connected with Grothendieck combinatorics are developed in \cite{MarbergScrimshaw2026,MarbergTongYu2026}. $K$-theoretic crystal operators on set-valued tableaux of rectangular shape are studied in \cite{PechenikScrimshaw2022}. The involutions $\iota$ and $\iota^*$ differ from these operators: their acting positions are determined by the first discrepancy from $T_0$ or $T_0^*$.

\section{Marked Gelfand--Tsetlin patterns and decorated five-vertex states}
\label{sec:five-vertex}

For a skew shape $\lambda/\mu$, the minimum entries and extra entries of a semistandard set-valued tableau determine an interlacing chain and a set of marks. The inner partition $\mu$ is the initial partition of the chain and imposes an additional condition on the marks. The same data determine a decorated five-vertex state whose lower and upper boundaries are prescribed by $\mu$ and $\lambda$, with this condition imposed on the nontrivial bumps.

Semistandard set-valued tableaux, marked Gelfand--Tsetlin patterns, and decorated five-vertex states are in bijection, with preservation of weight and excess. The tableau and lattice signatures are compared, and the ordinary type~A crystal operators correspond under these maps. The actions of $\iota$ and $\iota^*$ on marked patterns and decorated states are described explicitly. The commutation results of Sections~\ref{sec:iota-crystal} and~\ref{sec:duality} also hold on decorated states.

\subsection{Marked Gelfand--Tsetlin patterns for skew shapes}
\label{subsec:mgt}

Skew Gelfand--Tsetlin polytopes are studied in \cite{Alexandersson2016}, and classical Gelfand--Tsetlin patterns originate in \cite{GelfandTsetlin1950}. Marked Gelfand--Tsetlin patterns for set-valued tableaux of straight shape occur in \cite{MPS2021}. Free-fermionic realizations of skew stable Grothendieck polynomials appear in \cite{Iwao2022}.

A skew diagram is a \emph{horizontal strip} if it contains at most one box in each column. For partitions $\eta\subseteq\nu$, write
\[
\eta\prec\nu
\]
if $\nu/\eta$ is a horizontal strip. Equivalently,
\[
\nu_i\ge \eta_i\ge \nu_{i+1}
\qquad(i\ge1).
\]

For $T\in\SVT_n(\lambda/\mu)$ and $0\le k\le n$, define
\[
\lambda_{\min}^{(k)}(T)
=
\mu\cup\{u\in\lambda/\mu:T_u^{\min}\le k\}.
\]
\begin{lemma}
\label{lem:min-chain}
For every $T\in\SVT_n(\lambda/\mu)$,
\[
\mu=\lambda_{\min}^{(0)}(T)
\subseteq\lambda_{\min}^{(1)}(T)
\subseteq\cdots\subseteq
\lambda_{\min}^{(n)}(T)=\lambda,
\]
and every $\lambda_{\min}^{(k)}(T)$ is a partition. \[
\lambda_{\min}^{(k-1)}(T)\prec\lambda_{\min}^{(k)}(T)
\qquad(1\le k\le n).
\]
\end{lemma}

\begin{proof}
Fix $k$. Since the minimum entries are weakly increasing along rows, the boxes whose minimum entries do not exceed $k$, together with the boxes of $\mu$, form an initial segment in each row. Suppose that $(i+1,j)\in\lambda/\mu$ belongs to $\lambda_{\min}^{(k)}(T)$. If $(i,j)\in\lambda/\mu$,
\[
T_{i,j}^{\min}
\le T_{i,j}^{\max}
< T_{i+1,j}^{\min}
\le k.
\]
If $(i,j)\in\mu$, it already belongs to $\lambda_{\min}^{(k)}(T)$. Thus $\lambda_{\min}^{(k)}(T)$ is a Young diagram.

The minimum entries are strictly increasing down columns, hence two boxes in the same column cannot have the same minimum entry. It follows that $\lambda_{\min}^{(k)}(T)/\lambda_{\min}^{(k-1)}(T)$ contains at most one box in each column and is a horizontal strip.
\end{proof}

For an extra entry $k$ in tableau row $i$, the pair $(i,k)$ records its row and value. The minimum chain determines the box containing this entry.

\begin{lemma}
\label{lem:extra-entry-position}
Suppose that $k$ is an extra entry in the box $(i,j)$ of $T\in\SVT_n(\lambda/\mu)$. One has
\begin{equation}
\label{eq:extra-position}
j=(\lambda_{\min}^{(k-1)}(T))_i.
\end{equation}
\begin{equation}
\label{eq:mark-conditions}
(\lambda_{\min}^{(k)}(T))_{i+1}
<
(\lambda_{\min}^{(k-1)}(T))_i,
\qquad
(\lambda_{\min}^{(k-1)}(T))_i>\mu_i.
\end{equation}
\end{lemma}

\begin{proof}
Since $k$ is an extra entry, the minimum of $(i,j)$ is smaller than $k$, and $(i,j)$ belongs to $\lambda_{\min}^{(k-1)}(T)$. If $(i,j+1)\in\lambda/\mu$,
\[
T_{i,j+1}^{\min}\ge T_{i,j}^{\max}\ge k.
\]
Hence no box to the right of $(i,j)$ in row $i$ belongs to $\lambda_{\min}^{(k-1)}(T)$, which proves \eqref{eq:extra-position}.

Since $(i,j)$ lies outside $\mu$, one has $j>\mu_i$, and \eqref{eq:extra-position} proves the second inequality in \eqref{eq:mark-conditions}. If $(\lambda_{\min}^{(k)}(T))_{i+1}\ge j$, the box $(i+1,j)$ lies in $\lambda/\mu$, because $(i,j)\notin\mu$, and has minimum at most $k$. Since $T_{i,j}^{\max}\ge k$, this contradicts the strict column inequality. This proves the first inequality.
\end{proof}

\begin{definition}
\label{def:marked-gt}
A \emph{marked Gelfand--Tsetlin pattern for $\lambda/\mu$} consists of an interlacing chain
\[
\mu=\lambda^{(0)}
\prec\lambda^{(1)}
\prec\cdots\prec
\lambda^{(n)}=\lambda
\]
together with a set $M$ of marks $(i,k)$, with $i\ge1$ and $1\le k\le n$, satisfying
\begin{equation}
\label{eq:restricted-mark}
\lambda_{i+1}^{(k)}<\lambda_i^{(k-1)},
\qquad
\lambda_i^{(k-1)}>\mu_i
\end{equation}
for every $(i,k)\in M$.
\end{definition}

For $\mu=\varnothing$, the second inequality in \eqref{eq:restricted-mark} is automatic, and the first is the marking condition in the correspondence for straight shapes in \cite{MPS2021}. For a general skew shape, the additional inequality requires the box associated with a mark to lie outside the inner partition $\mu$.

\begin{proposition}
\label{prop:svt-mgt}
For $T\in\SVT_n(\lambda/\mu)$, put
$\lambda^{(k)}=\lambda_{\min}^{(k)}(T)$ for $0\le k\le n$, and let $M$ be the set of pairs $(i,k)$
for which $k$ occurs as an extra entry in row $i$. The chain $\lambda^{(0)}\prec\cdots\prec\lambda^{(n)}$, together with $M$, defines a bijection between
$\SVT_n(\lambda/\mu)$ and the marked Gelfand--Tsetlin patterns for
$\lambda/\mu$. Under this bijection,
\[
\ex(T)=|M|,
\]
and, for $1\le k\le n$,
\begin{equation}
\label{eq:mgt-content}
m_k(T)
=
|\lambda^{(k)}|-|\lambda^{(k-1)}|
+
\#\{i:(i,k)\in M\}.
\end{equation}
\end{proposition}

\begin{proof}
By Lemma~\ref{lem:min-chain}, the minimum entries determine the interlacing chain. Lemma~\ref{lem:extra-entry-position} identifies the admissible mark associated with each extra entry.

Conversely, assign the minimum entry $k$ to every box of $\lambda^{(k)}/\lambda^{(k-1)}$. The inclusion chain makes the minimum entries weakly increasing along rows. If two boxes in the same column had the same minimum entry $k$, the horizontal strip $\lambda^{(k)}/\lambda^{(k-1)}$ would contain two boxes in that column, a contradiction. Thus the minimum entries are strictly increasing down columns. For every mark $(i,k)$, add the entry $k$ to the box
\begin{equation}
\label{eq:mark-box}
(i,\lambda_i^{(k-1)}).
\end{equation}
The second inequality in \eqref{eq:restricted-mark} places this box outside $\mu$. Since it belongs to $\lambda^{(k-1)}$, its minimum entry is smaller than $k$.

A box immediately to the right of \eqref{eq:mark-box}, if present, does not belong to $\lambda^{(k-1)}$ and has minimum at least $k$. Adding $k$ preserves the row inequality. The first inequality in \eqref{eq:restricted-mark} implies that a box immediately below \eqref{eq:mark-box}, if present in $\lambda/\mu$, does not belong to $\lambda^{(k)}$ and has minimum larger than $k$. The strict column inequality is also preserved.

This filling lies in $\SVT_n(\lambda/\mu)$. Its boxes with minimum entry at most $k$ are the boxes of $\lambda^{(k)}/\mu$, which recovers the original chain. Lemma~\ref{lem:extra-entry-position} recovers the box corresponding to every mark. The two constructions are inverse. Every mark contributes one entry beyond the minimum of its box, and $\ex(T)=|M|$. Formula~\eqref{eq:mgt-content} follows by separating minimum and extra occurrences.
\end{proof}

\begin{example}
\label{ex:mgt-running}
For the tableau in Example~\ref{ex:running}, the minimum entries determine the chain
\[
(2,1)
\prec
(4,2,1)
\prec
(5,3,2)
\prec
(5,4,2)
\prec
(5,4,3),
\]
and the extra entries determine the marks
\[
(1,2),\qquad (2,3),\qquad (3,3).
\]
For the mark $(2,3)$, the associated box is
\[
(2,(\lambda_{\min}^{(2)}(T))_2)=(2,3),
\]
and the inequality involving the inner partition is
\[
(\lambda_{\min}^{(2)}(T))_2=3>\mu_2=1.
\]
The four horizontal strips have sizes $4,3,1,1$. The numbers of marks of levels $1,2,3,4$ are $0,1,2,0$. \[
\wt(T)=(4,4,3,1),
\qquad
\ex(T)=3.
\]
\end{example}

\subsection{Decorated five-vertex states for skew shapes}
\label{subsec:dfv}

Horizontal arrows point to the right and vertical arrows point upward. Lattice columns are numbered from left to right by positive integers, and the left boundary of each horizontal line is placed at lattice column $0$. Figure~\ref{fig:local-states} displays the local states and their weights on horizontal line $k$.

\begin{figure}[H]
\centering
\renewcommand{\arraystretch}{1.45}
\begin{tabular}{|c|c|c|c|c|c|}
\hline
Type & $a_1$ & $a_2$ & $b_1$ & $b_2$ & $c_1$\\
\hline
{\small Configuration}
& \raisebox{0pt}[2.2em][2.2em]{\localvertex{}{}}
& \raisebox{0pt}[2.2em][2.2em]{\localvertex{\draw[mutedred,-{Stealth[length=2.0mm]}] (0,-1.28)--(0,0)--(1.28,0);}{} }
& \raisebox{0pt}[2.2em][2.2em]{\localvertex{\draw[mutedblue,-{Stealth[length=2.0mm]}] (-1.28,0)--(1.28,0);\draw[mutedblue,-{Stealth[length=2.0mm]}] (0,-1.28)--(0,1.28);}{} }
& \raisebox{0pt}[2.2em][2.2em]{\localvertex{\draw[mutedblue,-{Stealth[length=2.0mm]}] (-1.28,0)--(1.28,0);}{} }
& \raisebox{0pt}[2.2em][2.2em]{\localvertex{\draw[mutedgreen,-{Stealth[length=2.0mm]}] (-1.28,0)--(0,0)--(0,1.28);}{} }\\
\hline
Weight & $1$ & $1+\beta x_k$ & $1$ & $x_k$ & $1$\\
\hline
\end{tabular}
\caption{The five local states on horizontal line $k$. Gray segments are unoccupied.}
\label{fig:local-states}
\end{figure}

The five local states, their weights, and the decoration of an $a_2$ vertex use the convention of \cite[Section~2.3]{JNS2026}, based on the models of \cite{MotegiSakai2013,MotegiSakai2014}.

\begin{definition}[\cite{JNS2026}]
\label{def:decorated-local-state}
An \emph{admissible state} is an assignment of occupied and unoccupied edges for which every lattice vertex has one of the five local states in Figure~\ref{fig:local-states}. A \emph{decorated state} is an admissible state with a choice of one summand in the weight of every $a_2$ vertex.
\end{definition}

An $a_2$ vertex is called a \emph{bump}. The choice $1$ is a trivial bump, and the choice $\beta x_k$ on horizontal line $k$ is a nontrivial bump. Trivial bumps are marked by open circles and nontrivial bumps by filled circles. The contributing vertices are the $b_2$ vertices and the nontrivial bumps.

For an interlacing chain in Definition~\ref{def:marked-gt}, put
\begin{equation}
\label{eq:particle-position}
p_i^{(k)}
=
\lambda_i^{(k)}+\ell(\mu)+k-i+1,
\qquad
1\le i\le\ell(\mu)+k.
\end{equation}

We call $k$ the \emph{level} of $\lambda^{(k)}$ and of a mark $(i,k)$, and call the vertical arrow in column $p_i^{(k)}$ the $i$th \emph{particle} at level $k$. We use path terminology for the occupied edges: an occupied horizontal edge is a \emph{horizontal step}, a consecutive portion of occupied edges is a \emph{path segment}, and moving such a segment by one lattice column is a \emph{path shift}.

Each horizontal strip adds at most one box in column $1$, and induction from $\lambda^{(0)}=\mu$ proves
\[
\ell(\lambda^{(k)})\le \ell(\mu)+k
\qquad(0\le k\le n).
\]
Together with $\lambda_i^{(k)}\ge0$, $\lambda_i^{(k)}\le\lambda_1$, and $i\le\ell(\mu)+k$, this places every position in \eqref{eq:particle-position} between $1$ and $\lambda_1+\ell(\mu)+n$.

The interlacing inequalities are equivalent to
\begin{equation}
\label{eq:particle-interlacing}
p_i^{(k)}>p_i^{(k-1)}\ge p_{i+1}^{(k)}.
\end{equation}

\begin{lemma}
\label{lem:row-correspondence}
For every interlacing pair $\lambda^{(k-1)}\prec\lambda^{(k)}$, there is a unique admissible configuration on horizontal line $k$ whose lower and upper vertical arrows occur at the positions $p_i^{(k-1)}$ and $p_i^{(k)}$. One occupied horizontal edge enters from the left, and the right boundary edge is unoccupied. Conversely, every such row configuration determines the interlacing pair.
\end{lemma}

\begin{proof}
Equation~\eqref{eq:particle-interlacing} interlaces the upper and lower arrow positions. At each lattice position the occupancies of the vertical edges are prescribed, and once the occupancy of the left horizontal edge is known, Figure~\ref{fig:local-states} determines the occupancy of the right horizontal edge uniquely. Equation~\eqref{eq:particle-interlacing} ensures admissibility, and the occupied left boundary edge determines the row from left to right. Each local state has the same number of occupied incoming and outgoing edges. Since there are $\ell(\mu)+k-1$ lower arrows and $\ell(\mu)+k$ upper arrows, the right boundary edge is unoccupied. Conversely, the positions of the vertical arrows of an admissible row satisfy the same interlacing inequalities, and \eqref{eq:particle-position} recovers the two partitions.
\end{proof}

\begin{lemma}
\label{lem:b2-count}
For $1\le k\le n$, horizontal line $k$ contains
\[
|\lambda^{(k)}|-|\lambda^{(k-1)}|
\]
vertices of type $b_2$.
\end{lemma}

\begin{proof}
The occupied horizontal edges on line $k$ decompose into segments. Their
left endpoints consist of the left boundary together with the positions of the $\ell(\mu)+k-1$ lower arrows.
Their right endpoints are the positions of the $\ell(\mu)+k$ upper arrows. With the left boundary at lattice column $0$, their total length is
\[
\sum_{i=1}^{\ell(\mu)+k}p_i^{(k)}
-\sum_{i=1}^{\ell(\mu)+k-1}p_i^{(k-1)}.
\]
The same quantity counts the vertices with an occupied left edge. Such a
vertex has type $b_1$, $b_2$, or $c_1$. The $b_1$ and $c_1$ vertices are
exactly the $\ell(\mu)+k$ vertices with an occupied upper edge. Using
\eqref{eq:particle-position},
\[
\#\{b_2\text{ vertices on line }k\}
=\sum_{i=1}^{\ell(\mu)+k}p_i^{(k)}
-\sum_{i=1}^{\ell(\mu)+k-1}p_i^{(k-1)}-(\ell(\mu)+k)
=|\lambda^{(k)}|-|\lambda^{(k-1)}|.
\]
\end{proof}

For a mark $(i,k)$, the corresponding lower vertical arrow lies in lattice column
\begin{equation}
\label{eq:mark-lattice-column}
p_i^{(k-1)}
=
\lambda_i^{(k-1)}+\ell(\mu)+k-i.
\end{equation}

\begin{lemma}
\label{lem:mark-vertex}
The vertex on horizontal line $k$ at the position $p_i^{(k-1)}$ of the lower arrow has type $a_2$ if and only if
\[
\lambda_{i+1}^{(k)}<\lambda_i^{(k-1)}.
\]
Equality corresponds to type $b_1$. The second condition in \eqref{eq:restricted-mark} is equivalent to
\begin{equation}
\label{eq:mu-boundary-bump}
p_i^{(k-1)}>
\mu_i+\ell(\mu)+k-i.
\end{equation}
\end{lemma}

\begin{proof}
The lower vertical edge at $p_i^{(k-1)}$ is occupied, and equation~\eqref{eq:particle-position} yields
\[
p_{i+1}^{(k)}=\lambda_{i+1}^{(k)}+\ell(\mu)+k-i.
\]
The strict inequality places $p_{i+1}^{(k)}$ strictly to the left of $p_i^{(k-1)}$, leaving the upper vertical edge at the marked column unoccupied and producing type $a_2$. Equality places an upper arrow in the same column and produces type $b_1$. Substitution of \eqref{eq:mark-lattice-column} into $\lambda_i^{(k-1)}>\mu_i$ yields \eqref{eq:mu-boundary-bump}.
\end{proof}

On horizontal line $k$, index the lower vertical arrows from right to left by
$ i=1,\ldots,\ell(\mu)+k-1$. Their positions are $p_i^{(k-1)}$.

\begin{definition}
\label{def:dfv}
A \emph{decorated five-vertex state for $\lambda/\mu$} is a decorated state on horizontal lines $1,\ldots,n$ and lattice width $\lambda_1+\ell(\mu)+n$ subject to conditions \textup{(i)}--\textup{(iii)}.
\begin{enumerate}
\item The bottom vertical arrows have positions
\[
\mu_i+\ell(\mu)-i+1
\qquad(1\le i\le\ell(\mu)),
\]
and the top vertical arrows have positions
\[
\lambda_i+\ell(\mu)+n-i+1
\qquad(1\le i\le\ell(\mu)+n).
\]
\item One occupied horizontal edge enters every horizontal line from the left, and no occupied horizontal edge exits to the right.
\item Every nontrivial bump at the $i$th lower vertical arrow on line $k$ satisfies \eqref{eq:mu-boundary-bump}.
\end{enumerate}
\end{definition}

The set of these states is denoted by $\DFV_n(\lambda/\mu)$. The first two conditions prescribe the boundary data. The third imposes the additional restriction determined by the inner partition $\mu$. It depends on the indexed lower arrow and the preceding partition $\lambda^{(k-1)}$, not only on the local vertex configuration. Thus $\DFV_n(\lambda/\mu)$ is a restricted class of decorated five-vertex states rather than a model defined solely by local Boltzmann weights. The restriction is exactly the second marking condition in Definition~\ref{def:marked-gt} under Proposition~\ref{prop:dfv-mgt}.

For $\mu=\varnothing$, the third condition is automatic, and Definition~\ref{def:dfv} reduces to the setting of decorated states for straight shapes in \cite{JNS2026}.

\begin{proposition}
\label{prop:dfv-mgt}
The positions of the vertical arrows and the nontrivial bumps define a bijection between $\DFV_n(\lambda/\mu)$ and the marked Gelfand--Tsetlin patterns for $\lambda/\mu$.
\end{proposition}

\begin{proof}
Lemma~\ref{lem:row-correspondence} identifies the vertical arrows on consecutive levels with an interlacing pair whose bottom and top levels are $\mu$ and $\lambda$. Each nontrivial bump lies at a unique lower vertical arrow on a unique horizontal line, and Lemma~\ref{lem:mark-vertex} supplies the two inequalities for the associated mark.

Conversely, the positions in \eqref{eq:particle-position} specify the admissible row configuration on every line. A mark $(i,k)$ selects the $a_2$ vertex in column \eqref{eq:mark-lattice-column}, with its second marking condition expressed by \eqref{eq:mu-boundary-bump}. Declaring this bump nontrivial and every other bump trivial recovers the decorated state, and the two constructions are inverse.
\end{proof}

\begin{proposition}
\label{prop:dfv-svt}
There is a bijection
\[
\DFV_n(\lambda/\mu)
\longleftrightarrow
\SVT_n(\lambda/\mu).
\]
For corresponding states $V$ and tableaux $T$,
\begin{align}
\prod_{v\in V}\operatorname{wt}(v)&=\beta^{\ex(T)}x^{\wt(T)},\label{eq:dfv-weight}\\
\ex(T)&=\#\{\text{nontrivial bumps of }V\}.\label{eq:dfv-excess}
\end{align}
\end{proposition}

\begin{proof}
The bijections in Propositions~\ref{prop:svt-mgt} and~\ref{prop:dfv-mgt} identify tableaux and decorated states with the same marked Gelfand--Tsetlin patterns. On line $k$, Lemma~\ref{lem:b2-count} counts $|\lambda^{(k)}|-|\lambda^{(k-1)}|$ vertices of type $b_2$, each of weight $x_k$. This equals the number of boxes whose minimum entry is $k$. A mark $(i,k)$ records one extra occurrence of $k$, and the corresponding nontrivial bump has weight $\beta x_k$. All other local weights are $1$. Formula~\eqref{eq:mgt-content} determines the exponent of $x_k$, and the number of marks is $\ex(T)$. Equations~\eqref{eq:dfv-weight} and \eqref{eq:dfv-excess} follow.
\end{proof}

\begin{example}
\label{ex:dfv-running}
For the tableau in Example~\ref{ex:running}, the particle positions are
\[
\begin{array}{c|l}
k&\{p_i^{(k)}\}\\
\hline
0&\{4,2\}\\
1&\{7,4,2\}\\
2&\{9,6,4,1\}\\
3&\{10,8,5,2,1\}\\
4&\{11,9,7,3,2,1\}
\end{array}.
\]
They determine the state in Figure~\ref{fig:running-dfv}. The filled circles are the three nontrivial bumps corresponding to the marks $(1,2),(2,3),(3,3)$.

\begin{figure}[H]
\centering
\begin{tikzpicture}[x=0.72cm,y=0.68cm,>=Stealth,line width=0.62pt]
  \foreach \x in {1,...,11}{\draw[faintgray] (\x,0)--(\x,4); \node[above] at (\x,4.05) {\scriptsize \x};}
  \foreach \y in {0.5,1.5,2.5,3.5}{\draw[faintgray] (0.25,\y)--(11.55,\y);}
  \foreach \y/\k in {0.5/1,1.5/2,2.5/3,3.5/4}{\node[left] at (0.12,\y) {\small $k=\k$};}
  \foreach \xa/\xb/\yy in {0.25/7/0.5,0.25/1/1.5,2/6/1.5,7/9/1.5,0.25/2/2.5,4/5/2.5,6/8/2.5,9/10/2.5,0.25/3/3.5,5/7/3.5,8/9/3.5,10/11/3.5}{
    \draw[pathgray,-{Stealth[length=1.55mm]}] (\xa,\yy)--(\xb,\yy);
  }
  \foreach \x in {2,4}{\draw[pathgray,-{Stealth[length=1.55mm]}] (\x,0)--(\x,0.5);}
  \foreach \x in {2,4,7}{\draw[pathgray,-{Stealth[length=1.55mm]}] (\x,0.5)--(\x,1.5);}
  \foreach \x in {1,4,6,9}{\draw[pathgray,-{Stealth[length=1.55mm]}] (\x,1.5)--(\x,2.5);}
  \foreach \x in {1,2,5,8,10}{\draw[pathgray,-{Stealth[length=1.55mm]}] (\x,2.5)--(\x,3.5);}
  \foreach \x in {1,2,3,7,9,11}{\draw[pathgray,-{Stealth[length=1.55mm]}] (\x,3.5)--(\x,4);}
  \foreach \x/\y in {2/1.5,9/2.5,5/3.5,8/3.5,10/3.5}{
    \fill[white] (\x,\y) circle (2.5pt); \draw[pathgray] (\x,\y) circle (2.5pt);
  }
  \foreach \x/\y in {7/1.5,4/2.5,6/2.5}{\fill[pathgray] (\x,\y) circle (2.6pt);}
\end{tikzpicture}
\caption{The decorated state corresponding to Example~\ref{ex:running}. Filled circles are nontrivial bumps and open circles are trivial bumps.}
\label{fig:running-dfv}
\end{figure}
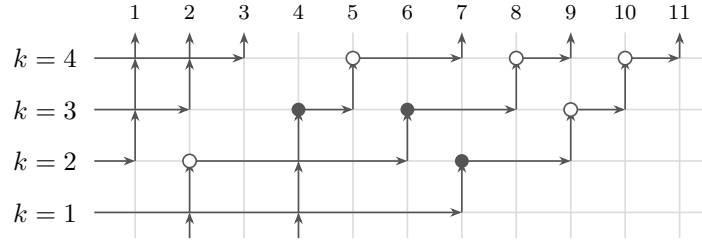

The contributing vertices on the four horizontal lines are
\[
\begin{array}{c|c|c|c}
k&b_2\text{ columns}&\text{nontrivial bump columns}&\text{line weight}\\
\hline
1&1,3,5,6&\varnothing&x_1^4\\
2&3,5,8&7&\beta x_2^4\\
3&7&4,6&\beta^2x_3^3\\
4&6&\varnothing&x_4
\end{array},
\]
whose product is
\[
\beta^3x_1^4x_2^4x_3^3x_4
=
\beta^{\ex(T)}x^{\wt(T)}.
\]
\end{example}

Summation of \eqref{eq:dfv-weight} over $V\in\DFV_n(\lambda/\mu)$, together with \eqref{eq:grothendieck}, yields
\[
\sum_{V\in\DFV_n(\lambda/\mu)}
\prod_{v\in V}\operatorname{wt}(v)
=
G_{\lambda/\mu}(x_1,\ldots,x_n;\beta).
\]

\subsection{Crystal operators on decorated states}
\label{subsec:dfv-crystal}

The lattice $r$-signature and the marked Gelfand--Tsetlin pattern define operators on $\DFV_n(\lambda/\mu)$. Crystal structures on five-vertex ice models are studied in \cite{LorcaEspiroVolk2017}. Proposition~\ref{prop:signature-compatibility} compares the lattice and tableau signatures, and Theorem~\ref{thm:crystal-intertwining} identifies these operators with the ordinary type~A crystal operators on $\SVT_n(\lambda/\mu)$. For $\mu=\varnothing$, these operators specialize to the transformations for straight shapes in \cite[Theorem~5.3]{JNS2026}. For a nonempty inner partition, the lower boundary and the restriction \eqref{eq:mu-boundary-bump} require additional verification.

\begin{lemma}
\label{lem:entry-lattice-coordinate}
Let $k\in T_{i,j}$. The contributing vertex corresponding to this occurrence lies on horizontal line $k$ in lattice column
\begin{equation}
\label{eq:entry-lattice-column}
j-i+\ell(\mu)+k.
\end{equation}
If $k=T_{i,j}^{\min}$, this vertex is of type $b_2$. If $k>T_{i,j}^{\min}$, it is a nontrivial bump.
\end{lemma}

\begin{proof}
If $k=T_{i,j}^{\min}$, $(i,j)$ belongs to $\lambda_{\min}^{(k)}(T)/\lambda_{\min}^{(k-1)}(T)$, and \eqref{eq:particle-position} places the corresponding horizontal step in column \eqref{eq:entry-lattice-column}. If $k>T_{i,j}^{\min}$, Lemma~\ref{lem:extra-entry-position} yields $j=(\lambda_{\min}^{(k-1)}(T))_i$, and \eqref{eq:mark-lattice-column} specifies the same column. The first occurrence is represented by a $b_2$ vertex and the second by a nontrivial bump.
\end{proof}

\begin{definition}
\label{def:lattice-signature}
For $1\le r<n$, the lattice $r$-signature uses horizontal lines $r$ and $r+1$. A contributing vertex on line $r$ contributes $+$, and a contributing vertex on line $r+1$ contributes $-$. First, a minus from a nontrivial bump is paired with an available plus one lattice column to its left, if such a plus exists. Among the signs not paired in this way, process the minuses from right to left and pair each with the nearest available plus weakly to its right. The unpaired signs form the reduced $r$-signature.
\end{definition}

\begin{proposition}
\label{prop:signature-compatibility}
For corresponding $T$ and $V$, the reduced tableau and lattice $r$-signatures have the same unpaired occurrences. The occurrence selected by $e_r$ or $f_r$ in $T$ corresponds to the selected vertex of $V$.
\end{proposition}

\begin{proof}
By Lemma~\ref{lem:entry-lattice-coordinate}, an occurrence in $(i,j)$ lies in lattice column
\[
j-i+\ell(\mu)+r
\]
on line $r$, and in lattice column
\[
j-i+\ell(\mu)+r+1
\]
on line $r+1$.
First consider the cancellations within a tableau column. If both values occur in one box, the occurrence of $r+1$ is extra and its minus lies one lattice column to the right of the plus. Conversely, a pair eligible for the first lattice pairing has equal tableau contents $j-i$. Distinct boxes with equal content are strictly northwest and southeast of one another. The skew diagram contains the rectangle between these boxes. If the occurrence of $r$ is northwest, the row and column inequalities force the minimum of the other box to be at least $r+1$, contrary to the assumption that its $r+1$ is extra. If the occurrence of $r$ is southeast, those inequalities force it to exceed $r+1$. The first lattice pairing removes precisely the pairs in a single box. Occurrences of a fixed value lie in strictly increasing lattice columns as their tableau columns increase: for $j<j'$, semistandardness forces $i\ge i'$, and $(j'-i')-(j-i)>0$. In particular, the pairs just described do not compete for a sign.

If $r$ and $r+1$ lie in distinct boxes of one tableau column, column strictness forces the boxes to be $(i,j)$ and $(i+1,j)$, respectively. Their lattice columns coincide. Neither sign was removed by the first pairing. In the second pairing, a minus strictly to their right cannot use this plus. At the turn of their minus, the plus is available and is the nearest one weakly to its right. All such pairs can be removed in advance without changing any other pairing.

After these cancellations, each remaining tableau column contributes at most one sign. Consider two surviving occurrences in boxes $(i,j)$ and $(i',j')$ with $j<j'$. It remains to prove $i\ge i'$. Otherwise, the rectangle between the boxes lies in the skew diagram. Since both entries belong to $\{r,r+1\}$, semistandardness forces the first occurrence to be $r$, the second to be $r+1$, and $i'=i+1$. The row and column inequalities force $T_{i,j'}=\{r\}$. Tableau column $j'$ contains both values, contrary to the survival of its sign. Thus $i\ge i'$.

Their lattice columns are strictly increasing. If the two occurrences have the same value, their difference is
\[
(j'-j)+(i-i')>0.
\]
If the first occurrence is $r$ and the second is $r+1$, the difference is
\[
(j'-j)+(i-i')+1>0.
\]
If the first occurrence is $r+1$ and the second is $r$, the difference is
\[
(j'-j)+(i-i')-1\ge0.
\]
Equality would require $j'=j+1$ and $i'=i$, which violates the row inequality. Hence the difference is positive in every case. All surviving signs, including opposite signs, occur in the same strict order in the tableau and lattice. The remaining pairing rules coincide, as do the unpaired occurrences and the occurrences selected by $e_r$ and $f_r$.
\end{proof}

\begin{definition}
\label{def:dfv-crystal-operators}
Let $V\in\DFV_n(\lambda/\mu)$, and let
\[
\mu=\lambda^{(0)}\prec\lambda^{(1)}\prec\cdots\prec\lambda^{(n)}=\lambda,
\qquad M,
\]
be its marked Gelfand--Tsetlin pattern under Proposition~\ref{prop:dfv-mgt}.

If the reduced $r$-signature has no unpaired minus, set $e_r(V)=0$. Otherwise select the leftmost unpaired minus on line $r+1$. If the selected vertex is a nontrivial bump corresponding to the mark $(i,r+1)$, leave the partition chain unchanged and replace $(i,r+1)$ in $M$ by $(i,r)$. If the selected vertex is of type $b_2$, let $i$ be the row of the corresponding box of $\lambda^{(r+1)}/\lambda^{(r)}$. Increase the $i$th part of $\lambda^{(r)}$ by one and leave every other partition and every mark unchanged.

If the reduced $r$-signature has no unpaired plus, set $f_r(V)=0$. Otherwise select the rightmost unpaired plus on line $r$. Suppose first that the selected vertex is a nontrivial bump corresponding to the mark $(i,r)$. If $(i,r+1)\in M$, decrease the $i$th part of $\lambda^{(r)}$ by one and leave every other partition and every mark unchanged. If $(i,r+1)\notin M$, leave the partition chain unchanged and replace $(i,r)$ in $M$ by $(i,r+1)$. If the selected vertex is of type $b_2$, let $i$ be the row of the corresponding box of $\lambda^{(r)}/\lambda^{(r-1)}$. Decrease the $i$th part of $\lambda^{(r)}$ by one and leave every other partition and every mark unchanged.
\end{definition}

Proposition~\ref{prop:raising-closure} proves closure under $e_r$, and Theorem~\ref{thm:crystal-intertwining} establishes closure under $f_r$ and the crystal intertwining.

\begin{example}
\label{ex:signature-running}
For the tableau in Example~\ref{ex:running} with $r=1$, tableau column $1$ contributes $+$ and column $5$ contributes $-$. Columns $2,3,4$ contain both values and contribute no sign. The occurrences of $1$ correspond to lattice columns $1,3,5,6$, and the occurrences of $2$ correspond to columns $3,5,7,8$. The pair in columns $6,7$ is removed by the first lattice pairing rule. The pairs in columns $3$ and $5$ are removed by the second. Figure~\ref{fig:signature-example} displays the surviving signs.

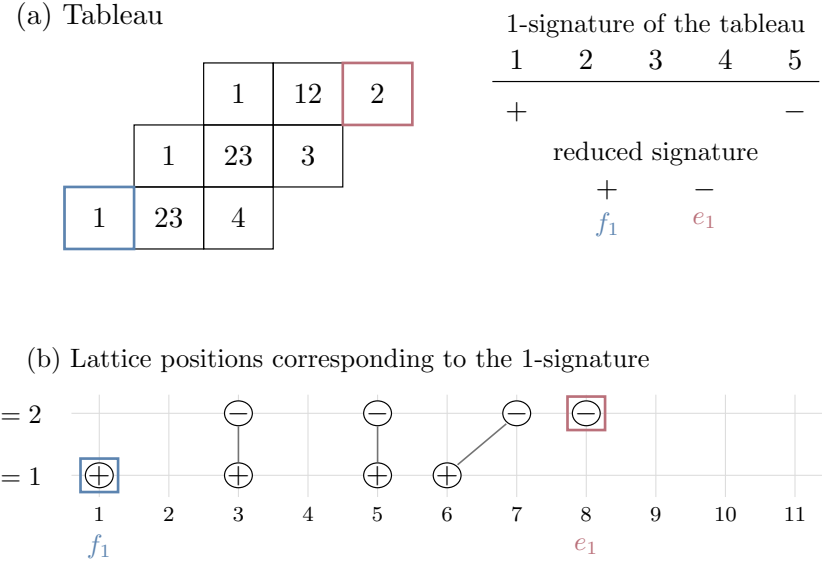
\begin{figure}[t]
\centering
\begin{tikzpicture}[x=0.92cm,y=0.82cm,>=Stealth]
\node[anchor=west,font=\normalsize] at (-0.35,4.25) {(a) Tableau};
\foreach \x/\y/\t in {3/3/1,4/3/12,5/3/2,2/2/1,3/2/23,4/2/3,1/1/1,2/1/23,3/1/4}{
  \draw[black,line width=0.45pt] (\x-0.5,\y-0.5) rectangle (\x+0.5,\y+0.5);
  \node at (\x,\y) {$\t$};
}
\draw[mutedblue,line width=1.0pt] (0.5,0.5) rectangle (1.5,1.5);
\draw[mutedred,line width=1.0pt] (4.5,2.5) rectangle (5.5,3.5);
\node[font=\small] at (9.0,4.10) {$1$-signature of the tableau};
\draw[black,line width=0.45pt] (6.65,3.20)--(11.35,3.20);
\foreach \x/\lab in {7/1,8/2,9/3,10/4,11/5}{\node at (\x,3.55) {$\lab$};}
\node at (7,2.72) {$+$}; \node at (11,2.72) {$-$};
\node[font=\small] at (9,2.05) {reduced signature};
\node at (8.30,1.45) {$+$}; \node at (9.70,1.45) {$-$};
\node[below,text=mutedblue] at (8.30,1.25) {\small $f_1$};
\node[below,text=mutedred] at (9.70,1.25) {\small $e_1$};
\begin{scope}[shift={(0,-3.15)}]
\node[anchor=west,font=\small] at (-0.20,1.85) {(b) Lattice positions corresponding to the $1$-signature};
\foreach \q in {1,...,11}{\draw[faintgray] (\q,-0.35)--(\q,1.30); \node[below] at (\q,-0.35) {\scriptsize \q};}
\draw[faintgray] (0.6,0)--(11.4,0); \draw[faintgray] (0.6,1)--(11.4,1);
\node[left] at (0.32,0) {\small $k=1$}; \node[left] at (0.32,1) {\small $k=2$};
\foreach \q in {1,3,5,6}{\fill[white] (\q,0) circle (0.205); \draw[black,line width=0.45pt] (\q,0) circle (0.20); \node at (\q,0) {$+$};}
\foreach \q in {3,5,7,8}{\fill[white] (\q,1) circle (0.205); \draw[black,line width=0.45pt] (\q,1) circle (0.20); \node at (\q,1) {$-$};}
\draw[black!55,line width=0.6pt] (3,0.21)--(3,0.79);
\draw[black!55,line width=0.6pt] (5,0.21)--(5,0.79);
\draw[black!55,line width=0.6pt] (6.15,0.17)--(6.85,0.83);
\draw[mutedblue,line width=1.0pt] (0.73,-0.27) rectangle (1.27,0.27);
\draw[mutedred,line width=1.0pt] (7.73,0.73) rectangle (8.27,1.27);
\node[text=mutedblue,font=\small] at (1,-1.15) {$f_1$};
\node[text=mutedred,font=\small] at (8,-1.15) {$e_1$};
\end{scope}
\end{tikzpicture}
\caption{The tableau and lattice $1$-signatures in Example~\ref{ex:signature-running}. The highlighted positions are selected by $f_1$ and $e_1$.}
\label{fig:signature-example}
\end{figure}
\end{example}

\begin{lemma}
\label{lem:left-box-raising}
Suppose that $e_r(T)\ne0$ selects the occurrence of $r+1$ in $(i,j)$. If $(i,j-1)\in\lambda/\mu$ and $r+1\in T_{i,j-1}$,
\[
r\in T_{i,j-1}.
\]
\end{lemma}

\begin{proof}
Assume that $r\notin T_{i,j-1}$. If tableau column $j-1$ contains no occurrence of $r$, it contributes a minus to the $r$-signature. The unpaired minus in column $j$ has no available plus to its right at its turn in the pairing. No plus lies between columns $j-1$ and $j$, and processing the minus in column $j$ frees no plus. The minus in column $j-1$ is also unpaired, contrary to the selection of the leftmost unpaired minus in column $j$. Thus column $j-1$ contains an occurrence of $r$.

Such an occurrence lies immediately above $(i,j-1)$. An occurrence below is excluded by column strictness, and an occurrence at least two rows above forces the minimum in $(i,j-1)$ to be at least $r+2$, contradicting $r+1\in T_{i,j-1}$. The box immediately above $(i,j-1)$ has maximum $r$. The box immediately above $(i,j)$ also belongs to $\lambda/\mu$, since $(i-1,j-1)$ and $(i,j)$ do. The row and column inequalities force this box to be $\{r\}$. Tableau column $j$ contains both $r$ and $r+1$ and contributes no minus, contrary to the selection by $e_r$. Thus $r\in T_{i,j-1}$.
\end{proof}

\begin{example}
\label{ex:three-crystal-cases}
Figure~\ref{fig:three-crystal-cases} displays the three raising configurations and the lowering transfer from a selected extra entry.
\begin{figure}[H]
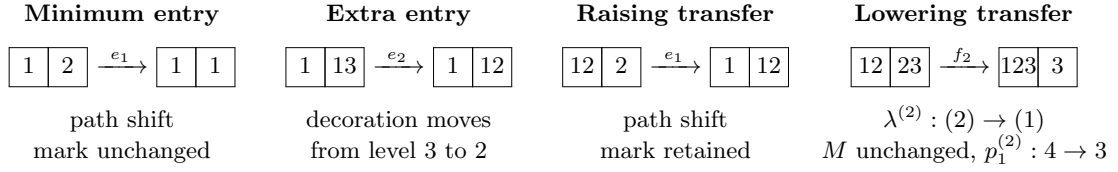

\centering
\footnotesize
\begingroup
\ytableausetup{mathmode,boxsize=1.55em,centertableaux}
\begin{tabular}{@{}>{\centering\arraybackslash}p{0.22\textwidth}@{\hspace{1.5mm}}>{\centering\arraybackslash}p{0.23\textwidth}@{\hspace{1.5mm}}>{\centering\arraybackslash}p{0.22\textwidth}@{\hspace{1.5mm}}>{\centering\arraybackslash}p{0.25\textwidth}@{}}
\textbf{Minimum entry} & \textbf{Extra entry} & \textbf{Raising transfer} & \textbf{Lowering transfer}\\[2mm]
$\begin{ytableau}1&2\end{ytableau}$
$\xrightarrow{\ e_1\ }$
$\begin{ytableau}1&1\end{ytableau}$
&
$\begin{ytableau}1&13\end{ytableau}$
$\xrightarrow{\ e_2\ }$
$\begin{ytableau}1&12\end{ytableau}$
&
$\begin{ytableau}12&2\end{ytableau}$
$\xrightarrow{\ e_1\ }$
$\begin{ytableau}1&12\end{ytableau}$
&
$\begin{ytableau}12&23\end{ytableau}$
$\xrightarrow{\ f_2\ }$
$\begin{ytableau}123&3\end{ytableau}$
\\[3mm]
\begin{tabular}[t]{@{}c@{}}\vphantom{$\lambda^{(2)}$}path shift\\[-0.2mm]\vphantom{$p_1^{(2)}$}mark unchanged\end{tabular}
& \begin{tabular}[t]{@{}c@{}}\vphantom{$\lambda^{(2)}$}decoration moves\\[-0.2mm]\vphantom{$p_1^{(2)}$}from level $3$ to $2$\end{tabular}
& \begin{tabular}[t]{@{}c@{}}\vphantom{$\lambda^{(2)}$}path shift\\[-0.2mm]\vphantom{$p_1^{(2)}$}mark retained\end{tabular}
& \begin{tabular}[t]{@{}c@{}}$\lambda^{(2)}:(2)\to(1)$\\[-0.2mm]$M$ unchanged, $p_1^{(2)}:4\to3$\end{tabular}
\end{tabular}
\endgroup
\caption{The three raising configurations and the lowering transfer from a selected extra entry. In the fourth panel, the selected nontrivial bump on line $2$ is retained and the nontrivial bump for $(1,3)$ moves one lattice column to the left.}
\label{fig:three-crystal-cases}
\end{figure}
\end{example}

Lemma~\ref{lem:e2-lower-boundary} verifies that the rule for a nontrivial bump preserves the restriction \eqref{eq:mu-boundary-bump}.

\begin{lemma}
\label{lem:e2-lower-boundary}
Suppose that the selected minus is a nontrivial bump corresponding to the mark $(i,r+1)$. The raising rule replaces this mark by $(i,r)$. Let $\lambda^{(0)}\prec\cdots\prec\lambda^{(n)}$ be the associated chain of the marked pattern. One has
\[
\lambda_i^{(r-1)}>\mu_i.
\]
\end{lemma}

\begin{proof}
The partition chain is unchanged under this rule. Assume equality. The box carrying the extra entry $r+1$ first appears at level $r$, hence its minimum entry is $r$ and the box contains both $r$ and $r+1$. Its tableau column contributes no sign to the $r$-signature, contradicting the selection of this occurrence by $e_r$.
\end{proof}

\begin{proposition}
\label{prop:raising-closure}
For $V\in\DFV_n(\lambda/\mu)$ and $1\le r<n$, every nonzero application of $e_r$ defined by the rules on marked patterns again belongs to $\DFV_n(\lambda/\mu)$ and has the same lower and upper boundary partitions.
\end{proposition}

\begin{proof}
Let $T$ correspond to $V$ under Proposition~\ref{prop:dfv-svt}. Proposition~\ref{prop:signature-compatibility} identifies the selected occurrence $r+1$ in a box $(i,j)$. The tableau crystal operator satisfies $e_r(T)\in\SVT_n(\lambda/\mu)$.

Suppose first that the selected $r+1$ is the minimum of $T_{i,j}$. If $(i,j-1)$ contains $r+1$, Lemma~\ref{lem:left-box-raising} implies $r\in T_{i,j-1}$, which is the transfer condition. Otherwise the transfer condition fails. In both cases only the $i$th part of the level-$r$ partition increases by one and the set of marks is unchanged. In a transfer, the extra $r+1$ moves one box to the right and the mark $(i,r+1)$ remains. These are the two $b_2$ configurations in Definition~\ref{def:dfv-crystal-operators}.

If the selected $r+1$ is extra, the row inequality excludes $r+1$ from the box immediately to the left. The transfer condition fails. Replacement leaves the partition chain fixed and changes the mark $(i,r+1)$ to $(i,r)$. This is the rule for a nontrivial bump in Definition~\ref{def:dfv-crystal-operators}, and Lemma~\ref{lem:e2-lower-boundary} verifies the restriction \eqref{eq:mu-boundary-bump} for the new mark.

For the two $b_2$ configurations, \eqref{eq:particle-position} moves only $p_i^{(r)}$ one column to the right. For the case involving a nontrivial bump, the positions of the vertical arrows are fixed and \eqref{eq:mark-lattice-column} shifts the nontrivial decoration from line $r+1$ to line $r$, one column to the left. In a transfer, the retained nontrivial bump on line $r+1$ moves one column to the right. By the uniqueness in Lemma~\ref{lem:row-correspondence}, these changes produce the state corresponding to $e_r(T)$, which belongs to $\DFV_n(\lambda/\mu)$ by Proposition~\ref{prop:dfv-svt}. Only the intermediate level $r$ can change, and the lower and upper boundary partitions remain $\mu$ and $\lambda$.
\end{proof}

\begin{theorem}
\label{thm:crystal-intertwining}
The bijection in Proposition~\ref{prop:dfv-svt} intertwines the ordinary type~A crystal operators on $\SVT_n(\lambda/\mu)$ with the operators on $\DFV_n(\lambda/\mu)$ defined by the rules on marked patterns. For corresponding $V$ and $T$,
\[
e_r(V)\longleftrightarrow e_r(T),
\qquad
f_r(V)\longleftrightarrow f_r(T).
\]
\end{theorem}

\begin{proof}
Proposition~\ref{prop:signature-compatibility} identifies the selected occurrence and shows that an operator vanishes on one side exactly if it vanishes on the other. The proof of Proposition~\ref{prop:raising-closure} establishes the raising relation by comparing the complete marked patterns.

For lowering, suppose $f_r(T)\ne0$ selects $r$ in $u=(i,j)$. If the selected $r$ is the minimum of $T_u$, the selected vertex is of type $b_2$. In a replacement, the minimum of $u$ changes from $r$ to $r+1$. In a transfer, the box $(i,j+1)$ has minimum $r$, removal of $r$ changes this minimum to $r+1$, and the minimum of $u$ remains $r$. In both cases the $i$th part of the level-$r$ partition decreases by one and every mark is unchanged. This is the $b_2$ lowering rule in Definition~\ref{def:dfv-crystal-operators}.

Suppose that the selected $r$ is extra. The selected vertex is the nontrivial bump corresponding to $(i,r)$. By Lemma~\ref{lem:extra-entry-position},
\[
j=(\lambda_{\min}^{(r-1)}(T))_i.
\]
The transfer condition is equivalent to $(i,r+1)\in M$. A transfer has $r+1$ as an extra entry of $(i,j+1)$ before the move, recorded by the mark $(i,r+1)$. Conversely, assume $(i,r+1)\in M$. Its box is $(i,(\lambda_{\min}^{(r)}(T))_i)$. The selected column contains $r$ but not $r+1$, and $(\lambda_{\min}^{(r)}(T))_i\ge j+1$. If $(\lambda_{\min}^{(r)}(T))_i\ge j+2$, the boxes $(i,j+1)$ and $(i,j+2)$ both have minimum $r$, since $(\lambda_{\min}^{(r-1)}(T))_i=j$. The row inequality forces $T_{i,j+1}=\{r\}$. If its column contributed a plus, any minus pairing it would first pair with the unpaired plus in column $j$. The column must contain $r+1$ below row $i$. The box $(i+1,j)$ belongs to $\lambda/\mu$, because $(i+1,j+1)\in\lambda/\mu$ and $(i,j)\notin\mu$. Column strictness and the row inequality force $T_{i+1,j}=\{r+1\}$, contradicting the absence of $r+1$ from the selected column. This contradiction forces $(\lambda_{\min}^{(r)}(T))_i=j+1$, and $(i,j+1)$ contains both $r$ and $r+1$. The transfer condition holds.

If $(i,r+1)\in M$, the transfer decreases the $i$th part of the level-$r$ partition by one and leaves $M$ unchanged. If $(i,r+1)\notin M$, no transfer occurs. Replacing the extra entry $r$ by $r+1$ leaves the partition chain unchanged and replaces the mark $(i,r)$ by $(i,r+1)$. These are exactly the two rules for a selected nontrivial bump in Definition~\ref{def:dfv-crystal-operators}.

The chain and marks obtained above are the marked Gelfand--Tsetlin pattern of $f_r(T)$. They are admissible by Proposition~\ref{prop:svt-mgt}. Proposition~\ref{prop:dfv-mgt} identifies a state in $\DFV_n(\lambda/\mu)$. This proves the lowering relation and closure under $f_r$.
\end{proof}

By Lemma~\ref{lem:crystal-excess}, Proposition~\ref{prop:dfv-svt}, and Theorem~\ref{thm:crystal-intertwining}, the number of nontrivial bumps is constant along every nonzero crystal edge.

\subsection{The involutions on marked patterns and decorated states}
\label{subsec:iota-dfv}

Through the bijections in Propositions~\ref{prop:svt-mgt} and
\ref{prop:dfv-svt}, the tableau involutions $\iota$ and $\iota^*$ act on marked
Gelfand--Tsetlin patterns and decorated states. The same symbols denote these induced maps.

For $T\ne T_0$, write $u=u_\iota(T)=(i,j)$ and
$k_0=i-\mu'_j$, and set
\begin{equation}
\label{eq:k1}
k_1=\min\bigl(T_u\setminus\{k_0\}\bigr).
\end{equation}
The proof of Lemma~\ref{lem:iota-max-preservation} shows that every entry of
$T_u$ is at least $k_0$. If $k_0\notin T_u$, the minimum is larger than
$k_0$. If $k_0\in T_u$, the discrepancy $T_u\ne\{k_0\}$ leaves another
entry larger than $k_0$. Hence the set in \eqref{eq:k1} is nonempty and
$k_0<k_1\le n$.

\begin{proposition}
\label{prop:iota-mgt}
Assume that $k_0\notin T_u$. One has
\begin{equation}
\label{eq:iota-growth-change}
\bigl(\lambda_{\min}^{(k)}(\iota(T))\bigr)_i
=
\begin{cases}
\bigl(\lambda_{\min}^{(k)}(T)\bigr)_i+1,&k_0\le k<k_1,\\
\bigl(\lambda_{\min}^{(k)}(T)\bigr)_i,&\text{otherwise}.
\end{cases}
\end{equation}
Every other part of every partition is unchanged. The mark $(i,k_1)$
is added, and every other mark is unchanged. If $k_0\in T_u$, the
involution reverses these changes.
\end{proposition}

\begin{proof}
Under the assumption $k_0\notin T_u$, the selected box has minimum $k_1$ in
$T$ and minimum $k_0$ in $\iota(T)$. It belongs to
$\lambda_{\min}^{(k)}(\iota(T))$ and does not belong to
$\lambda_{\min}^{(k)}(T)$ for $k_0\le k<k_1$. No other minimum
changes. The occurrence $k_1$ is a minimum entry in $T$ and an extra
entry in $\iota(T)$, producing the mark $(i,k_1)$. Removal reverses
these changes.
\end{proof}

For $k_0\notin T_u$, equations~\eqref{eq:particle-position} and
\eqref{eq:iota-growth-change} move the $i$th particle one lattice
column to the right for $k_0\le k<k_1$. The associated path segment
is displaced by one column. A $b_2$ vertex appears on line $k_0$, and
the $b_2$ vertex recording the old minimum $k_1$ becomes a
nontrivial bump on line $k_1$. Removal reverses the displacement.
This path modification represents the Boolean coordinate change in
Proposition~\ref{prop:iota-coordinate}. It is distinct from the local crystal transformations in Subsection~\ref{subsec:dfv-crystal}.

\begin{example}
\label{ex:iota-dfv}
For the tableau in Example~\ref{ex:running}, the first discrepancy is
$(3,2)$, with $k_0=2$ and $k_1=3$. This is the reverse direction of the
change displayed in Proposition~\ref{prop:iota-mgt}: the involution removes
$2$ from that box. Only level $2$ changes in the sequence of partitions,
from $(5,3,2)$ to $(5,3,1)$, and the mark $(3,3)$ disappears. On the
lattice, the associated path segment shifts one column to the left on
line $2$, and the nontrivial bump on line $3$ becomes a $b_2$ vertex.
\end{example}

\begin{proposition}
\label{prop:iota-star-mgt}
Let $T\ne T_0^*$, let $u=u_{\iota^*}(T)=(i,j)$, and write
$d^*=d^*(u)$. The marked Gelfand--Tsetlin patterns of $T$ and
$\iota^*(T)$ have the same interlacing chain. If $d^*\notin T_u$, the
mark $(i,d^*)$ is added. If $d^*\in T_u$, this mark is removed. Every
other mark is unchanged.

Under the correspondence with decorated states, all occupied edges are unchanged.
Adding the mark changes the bump associated with $(i,d^*)$ from trivial to nontrivial.
Removing the mark changes it from nontrivial to trivial.
\end{proposition}

\begin{proof}
Proposition~\ref{prop:iota-star} shows that $\iota^*$ preserves every minimum entry, hence the interlacing chain is unchanged. Only the occurrence of $d^*$ in the box $u$ is inserted or removed, with the minimum of $u$ fixed, and every such occurrence is extra. Proposition~\ref{prop:svt-mgt} adds or removes precisely the mark $(i,d^*)$.

By Proposition~\ref{prop:dfv-mgt}, the positions of the vertical arrows depend only on the interlacing chain and are unchanged. The presence of the mark makes the associated bump nontrivial, and its absence makes the same bump trivial. This proves the stated change of decoration.
\end{proof}

Let $T$ correspond to $V$. If $T\ne T_0$ and $k_0$ is determined by its
first discrepancy, Theorem~\ref{thm:crystal-intertwining} and
Theorem~\ref{thm:raising-away} imply, for $r\ne k_0$,
\[
e_r\iota(V)=\iota e_r(V).
\]
Proposition~\ref{prop:lowering-generic} states the lowering criterion, and Theorem~\ref{thm:exceptional}
treats the exceptional raising color on decorated states.

If $T\ne T_0^*$, let $d^*=d^*(u_{\iota^*}(T))$. By
Theorem~\ref{thm:crystal-intertwining} and Corollary~\ref{cor:iota-star-lowering},
for $r\ne d^*-1$,
\[
f_r\iota^*(V)=\iota^*f_r(V).
\]
Corollary~\ref{cor:iota-star-raising} states the raising criterion, with the
selected box read in reverse column order. Corollary~\ref{cor:iota-star-exceptional}
treats the exceptional lowering color $d^*-1$.

In the exceptional commuting case of Example~\ref{ex:exceptional-square},
the raising crystal edge before applying $\iota$ is a replacement, and the edge
after applying $\iota$ is a transfer. This distinction from
Section~\ref{sec:iota-crystal} appears in the decorated model.
For $\iota^*$, the exceptional lowering crystal edges on the two sides of the
commutation relation in Corollary~\ref{cor:iota-star-exceptional} are the
corresponding rotated cases. The model on decorated states records the exceptional
phenomena described above together with their duals under rotation.

\FloatBarrier

\section{Conclusion}
\label{sec:conclusion}

Fibers determined by maximum entries provide Boolean coordinates for the sign-reversing involution, and Section~\ref{sec:iota-crystal} determines its interaction with the ordinary type~A crystal operators. The generic raising colors commute with $\iota$, the generic lowering colors are characterized by the position selected by the reduced signature, and Theorem~\ref{thm:exceptional} treats the exceptional raising color. Rotation produces the dual statements for $\iota^*$. The lowering colors involving the changed entry, and the dual raising colors obtained from them, are not characterized in this paper.

Marked Gelfand--Tsetlin patterns and restricted decorated five-vertex states realize the tableau structures for skew shapes. The inner partition enters through the lower boundary and through the additional restriction on nontrivial bumps. Two questions remain. One is whether the untreated crystal colors admit criteria comparable to Theorem~\ref{thm:exceptional}. The other is whether the restriction on nontrivial bumps for a skew shape can be encoded by a local vertex construction that retains the tableau and crystal correspondences. Related questions may also be considered for shifted set-valued tableaux and $K$-theoretic Schur $P$- and $Q$-functions \cite{IkedaNaruse2013,LewisMarberg2021,MarbergTong2025,NobukawaShimazaki}.

\section*{Acknowledgments}
This work was supported by JSPS KAKENHI Grant Number JP26K24505.
The author acknowledges the use of ChatGPT to assist with mathematical arguments and language editing.

\end{document}